\documentclass[12pt]{amsart}
\usepackage{amsfonts, amssymb, amsmath, amscd, amsthm, txfonts, graphicx, color, enumerate}  
 \usepackage{hyperref}

 \allowdisplaybreaks

\theoremstyle{plain}
\newtheorem{teo}{Theorem}[section]
\newtheorem{propo}[teo]{Proposition}
\newtheorem{lem}[teo]{Lemma}
\newtheorem{cor}[teo]{Corollary}

\theoremstyle{definition}
\newtheorem{defin}[teo]{Definition}
\newtheorem{ej}[teo]{Example}

\newtheorem{obs}[teo]{Remark}

 \newtheorem{que}[teo]{Question}
 
\usepackage{tikz} 

\numberwithin{equation}{section}

\numberwithin{figure}{section}

\begin{document} 
%\runningtitle{Structural Stability and a characterization  of  Anosov Families}
\title{Structural Stability and a characterization of Anosov Families}
 
\author{Jeovanny de Jesus Muentes Acevedo}
\address{Grupo de Investigaci\'on Deartica, Universidad del Sin\'u\\ 
El\'ias Bechara Zain\'um\\
Cartagena, Colombia}  \email{jeovanny.muentes@unisinu.edu.co}
 
%\authorheadline{J. Muentes}

%\support{CAPES \& CNPq}

\begin{abstract}
Anosov families are non-stationary dynamical systems with hyperbolic behavior. Non-trivial
examples of Anosov families will be given in this paper. We show the existence of invariant manifolds, the structrural stability
and a characterization for a certain class of Anosov families.  
 \end{abstract}

 \subjclass[2010]{  	37D20;   37C75; 37B55}
%% Four or five keywords or phrases
\keywords{Anosov families, Anosov diffeomorphism,    structural stability, non-stationary dynamical systems, non-autonomous dynamical systems}

\maketitle
 
\section{Introduction}
 
 Anosov families were introduced by P. Arnoux and A. Fisher in \cite{alb}, motivated by generalizing the
notion of Anosov diffeomorphisms. Roughly, an Anosov family is a non-stationary dynamical system
 $ (f_{i})_{i\in\mathbb{Z}}$ defined on a sequence of compact Riemannian manifolds $(M_{i})_{i\in\mathbb{Z}}$, which has a similar behavior
to an Anosov diffeomorphisms (see Definition \ref{anosovfamily}). It is important to point out that there exist Anosov
families $(f_{i})_{i\in\mathbb{Z}}$  such that the $f_{i}$'s  are not necessarily Anosov diffeomorphisms (see \cite{alb}, Example 3).
Furthermore, the $M_{i}$'s, although they are diffeomorphic, they are not necessarily isometric, thus, the
hyperbolicity could be induced by the Riemannian metrics (see \cite{alb}, \cite{Jeo2}  for more detail).

\medskip

Let \( \mathcal{M}\) be the disjoint union of the family of Riemannian manifolds  $M_{i}$, for $i\in\mathbb{Z}$. For $m\geq1$, let      $\mathcal{D}^{m}(\mathcal{M})$  be the set consisting of the families of $C^{m}$-diffeomorphisms on $\mathcal{M}$, which is endowed      with the \textit{strong topology} (see  Section 2) or with the \textit{uniform topology}.     We denote by $\mathcal{A}^{m}(\mathcal{M})$ the subset of $\mathcal{D}^{m}(\mathcal{M})$ consisting of Anosov families and by $\mathcal{A}^{m}_{b}(\mathcal{M})$  the set consisting of $C^{m}$ Anosov families  with bounded second derivative and  such that the angles between the unstable and stable subspaces are bounded away from 0 (see Definition \ref{propang2}).  Young, in  \cite{young}, Proposition 2.2, proved that     families consisting of  $C^{1+1}$ random small perturbations of an Anosov diffeomorphism of class $C^{2}$ are   Anosov.  In \cite{Jeo1} we prove    for any $(f_{i})_{i\in\mathbb{Z}}$, there exists a sequence of positive numbers  $(\varepsilon_{i})_{i\in\mathbb{Z}}$ such that, if $(g_{i})_{i\in\mathbb{Z}}\in \mathcal{D}^{1}(\mathcal{M})$, such that  $g_{i}$ is $\varepsilon_{i}$-close to $f_{i}$ in the $C^{1}$-topology, then $(g_{i})_{i\in\mathbb{Z}}\in \mathcal{A}^{1}(\mathcal{M})$.   This fact means that $\mathcal{A}^{1}(\mathcal{M})$  is an open subset of $\mathcal{D}^{1}(\mathcal{M})$  endowed
with the strong topology. The most important implication of this result is the great variety of non-trivial examples that it provides (non-trivial examples of Anosov families can be found in \cite{alb} and \cite{Jeo2}, thus the result in \cite{Jeo1} proves that, in a certain way, these examples are not isolated),  since we only ask
that the family be Anosov and we do not ask for any additional condition. Considering the uniform
topology on $\mathcal{D}^{m}(\mathcal{M})$, we will prove that $\mathcal{A}^{2}_{b}(\mathcal{M})$ is an open subset of $\mathcal{D}^{2}(\mathcal{M})$, which generalizes the Young's result.  

\medskip

Non-stationary dynamical systems are classified by \textit{uniform conjugacy} (see  Definition \ref{definconjugacy}).  The notion of uniform conjugacy   is also known as
\textit{equi-conjugacy} in the literature (see, e.g., \cite{Kolyada}). This   conjugacy  is  also considered to characterize \textit{random dynamical systems} and \textit{non-autonomous dynamical systems} or \textit{time-dependient dynamical systems} (see  \cite{Jeo3}, \cite{Liu}).  Kolyada and Snoha introduced a notion of topological entropy   for non-stationary dynamical systems (see \cite{Kolyada}, \cite{Kolyada2}, \cite{Jeo3}). This entropy is invariant by uniform conjugacy. C. Kawan and Y. Latushkin, in \cite{KawanL}, gave formulas for the
entropy of \textit{non-stationary subshifts of finite type}, introduced by Fisher and Arnoux in \cite{alb}, which code non-stationary dynamical systems, and in particular Anosov families. In \cite{Jeo3} we proved that for any sequence of $C^{1}$-diffeomorphisms $(f_{i})_{i\in\mathbb{Z}}$, there exists a sequence of positive numbers $(\delta_{i})_{i\in\mathbb{Z}}$ such that if $g_{i}$ is a $C^{1}$-diffeomorphism which is $\delta_{i}$-close to $f_{i}$ in the $C^{1}$-topology, then $(f_{i})_{i\in\mathbb{Z}}$ and $(g_{i})_{i\in\mathbb{Z}}$ have the same topological entropy.

 \medskip
 
 If $E$ is a vector bundle over a compact Riemannian manifold $M$, set $\Gamma(E)$ the Banach $  \mathbb{R}$-vector
space of continuous sections of $E$ over $M$, endowed with the $C^{0}$-topology. An automorphism $L :
E \rightarrow E$ is called \textit{hyperbolic} if 
$\sigma(L) \cap \mathbb{S}^{1} =\emptyset$, where $\sigma(L) $ is the spectrum of $L$. For $p \in M$, let $E_{p}$
denote the fiber of $E$ over $p$. For a diffeomorphism $f$ on $M$, let $f_{\ast} : \Gamma(T M) \rightarrow \Gamma(T M)$ be the bounded
linear operator given by $f_{\ast} (\zeta) = D f \circ \zeta \circ f ^{-1}$, where $T M$ is the tangent bundle over $M$. J. Mather in
\cite{Mather1}, \cite{Mather12} proved that $f$ is an Anosov diffeomorphism if and only if $f_{\ast}$ is a hyperbolic automorphism.
We can define the operator $f_{\ast}$ for non-stationary dynamical systems. In this case, that operator could be unbounded or non-hyperbolic (see Section 4). We will give some conditions on an Anosov
family to obtain the hyperbolicity of $f_{\ast}$ defined for the family (see Theorem \ref{characte}).

\medskip

 Structural stability for non-stationary dynamical systems with respect to the uniform topology
on $\mathcal{D}^{m}(\mathcal{M})$ will be defined in Section 2 (see Definition \ref{estruturalmenteestavel}).    A. Castro, F. Rodrigues and P. Varanda, in \cite{Acastro}, Theorem 2.3, proved the stability of  sequences of Anosov diffeomorphisms which are  small perturbations of a fixed Anosov diffeomorphism. Liu in  \cite{Liu}, Theorem 1.1,  proved this same fact for the random case.  In this work we
will show that $\mathcal{A}^{2}_{b}(\mathcal{M})$ is structurally stable in $\mathcal{D}^{1}(\mathcal{M})$, which generalizes the above results, since the elements in $\mathcal{A}^{2}_{b}(\mathcal{M})$ are not necessarily small perturbations of a fixed diffeomorphism. %\blue{On the other hand,   Another approach on the stability of non-stationary hyperbolic dynamical systems can be found in \cite{Acastro}  and \cite{Jfranks}. On p. 2, the author mentions earlier approaches to the stability of non- autonomous hyperbolic dynamical systems. However, a comparison of these approaches to the one pursued by the author is missing.}

\medskip

In the next section we will define the class of objects to be studied in this work: Anosov families.
 Furthermore, we will introduce the strong and uniform topologies on $\mathcal{D}^{m}(\mathcal{M})$ and the uniform conjugacy to classify non-stationary dynamical systems. In Section 3 we will show some results that
provide a great variety of examples of Anosov families. Another examples and properties of Anosov
families can be found in \cite{alb},  \cite{Jeo2} and \cite{Jeo3}. A characterization of Anosov families will be given in Section
4, which generalizes the characterization of J. Mather in \cite{Mather1} and \cite{Mather12} for Anosov diffeomorphisms. %We will introduce  the strong topology    on $\mathcal{D}^{m}(\mathcal{M})$ and \red{the conjugacy  which is used to classify non-stationary dynaical systems}. 
 In Section 5 we will  prove the openness of $\mathcal{A}^{2}_{b}(\mathcal{M})$ in $\mathcal{D}^{2}(\mathcal{M})$ with respect to the uniform topology.
We will see in Section 6 that each family in $\mathcal{A}^{2}_{b}(\mathcal{M})$ admits stable and unstable at every point of $\mathcal{M}$.    Finally, the structural stability of  $\mathcal{A}^{2}_{b}(\mathcal{M})$ will be proved in Section 7.

  \section{Anosov families}

In this work we will consider  a sequence of Riemannian manifolds   $M_{i}$   with fixed Riemannian metric $\langle \cdot, \cdot\rangle_{i}$ for $i\in \mathbb{Z}$. Consider the  \textit{disjoint union}  $$\mathcal{M}=\coprod_{i\in \mathbb{Z}}{M_{i}}=\bigcup_{i\in \mathbb{Z}}{M_{i}\times{i}}.$$ %The set $\textbf{M}$ will be called   \textit{total space} and the $M_{i}$ will be called  \textit{components}.   
   $\mathcal{M}$ will be endowed with the Riemannian metric $\langle \cdot, \cdot\rangle$    induced by  $\langle \cdot, \cdot\rangle_{i} $, setting 
\begin{equation}\label{metricariemannaian} 
\langle \cdot, \cdot\rangle|_{M_{i}}=\langle \cdot, \cdot\rangle_{i} \quad\text{ for }i\in \mathbb{Z}.
\end{equation} 
  
 We denote by $\Vert \cdot\Vert_{i}$ the induced norm by  $\langle\cdot,\cdot\rangle_{i}$ on $TM_{i}$ and we will take   $\Vert \cdot \Vert$ defined on  $\mathcal{M}$  as    $\Vert \cdot\Vert|_{M_{i}}=\Vert \cdot\Vert_{i} $ for $i\in \mathbb{Z}$. If $d_{i}(\cdot,\cdot)$ is the metric on $M_{i}$ induced by $\langle \cdot, \cdot\rangle_{i}$, then $\mathcal{M}$ is endowed with the metric
\begin{equation*}\label{primerametric} 
d(x,y)=    
        \begin{cases}
        \min\{1,d_{i}(x,y)\} & \mbox{if }x,y\in M_{i} \\
        	1  & \mbox{if }x\in M_{i}, y\in M_{j} \mbox{ and }i\neq j.  \\
        \end{cases} 
\end{equation*}

\begin{defin}\label{leidecomposicao} A  \textit{non-stationary dynamical system}  $(\mathcal{M},\langle\cdot,\cdot\rangle, \mathcal{F})$ (or \textit{n.s.d.s}, for short)  is a map $\mathcal{F}:\mathcal{M}\rightarrow \mathcal{M}$, such that, for each $i\in\mathbb{Z}$, $\mathcal{F}|_{M_{i}}=f_{i}:M_{i}\rightarrow M_{i+1}$ is a  diffeomorphism. Sometimes we use the notation   $\mathcal{F}=(f_{i})_{i\in\mathbb{Z}}$. The composition law is defined   to be   
\begin{equation*}
\mathcal{F}_{i} ^{\, n}:= 
\begin{cases}
f_{i+n-1}\circ \cdots\circ f_{i}:M_{i}\rightarrow M_{i+n}  & \mbox{if }n>0 \\
  f_{i+n}^{-1}\circ \cdots\circ f_{i-1}^{-1}:M_{i}\rightarrow M_{i+n} & \mbox{if } n<0  \\
	 I_{i}:M_{i}\rightarrow M_{i}  & \mbox{if } n=0,\\
 \end{cases} 
\end{equation*}
where $I_{i}$ is the identity on $M_{i}$  (see Figure \ref{Toros12}). 
\begin{figure}[h]
\begin{center}
\begin{tikzpicture}
\tikzstyle{point}=[circle,thick,draw=black,fill=black,inner sep=0pt,minimum width=1pt,minimum height=1pt]
\newcommand*{\xMin}{0}%
\newcommand*{\xMax}{6}%
\newcommand*{\yMin}{0}%
\newcommand*{\yMax}{6}%

\draw (-6.3,0.2) node[below] {\small \dots};
  
\draw (-6,0) .. controls (-5.9,0.8) and (-3.1,0.8) .. (-3,0);
\draw (-5.86,-0.01) .. controls (-5.85,0.6) and (-3.15,0.6) .. (-3.14,-0.01);
\draw (-6,0) .. controls (-5.9,-0.8) and (-3.1,-0.8) .. (-3,0);
\draw (-5.88,0.05) .. controls (-5.85,-0.6) and (-3.15,-0.6) .. (-3.12,0.05);
\draw (-4.5, -0.9) node[below] {\small $M_{i-1}$};

\draw (-2.5,0.45) node[below] {\small $\xrightarrow{f_{i-1}}$};

\draw (-2,0) .. controls (-2.05,0.8) and (0.1,0.8) .. (0,0);
\draw (-2,0) .. controls (-2.1,-0.8) and (0.1,-0.8) .. (0,0);
\draw (-1.5,-0.029) .. controls (-1.31,0.2) and (-0.7,0.2) .. (-0.5,-0.03);
\draw (-1.55,0.05) .. controls (-1.45,-0.2) and (-0.55,-0.2) .. (-0.45,0.05);

\draw (-1, -0.9) node[below] {\small $M_{i}$};

\draw (0.5,0.45) node[below] {\small  $\xrightarrow{\,\, \, f_{i}\, \, \,}$};

 \draw (1,0) .. controls (1,-0.7) and (3.5,-1.5) .. (3.5,0);
 \draw (1.2,0.1) .. controls (1,-0.3) and (3,-0.1) .. (3,0.4);

 \draw (1,0) .. controls (1,0.7) and (3.5,1.5) .. (3.5,0);
 \draw (1.2,0) .. controls (1.1,0.3) and (3,0.5) .. (2.8,0.15);
\draw (2.5, -0.9) node[below] {\small $M_{i+1}$};

 \draw (3.9,0.2) node[below] {\small \dots};
\end{tikzpicture}
\end{center}
\caption{A non-stationary dynamical system on a sequence of 2-torus with different Riemannian metrics.}
\label{Toros12}
\end{figure}
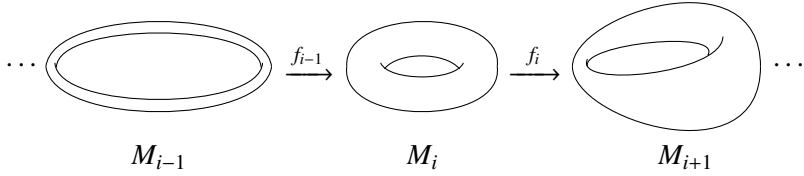
\end{defin}

We use the notation $(\mathcal{M}, \langle\cdot,\cdot\rangle,\mathcal{F})$, to indicate that we are fixing the Riemannian metric given in \eqref{metricariemannaian}.

\medskip

The study of time-depending dynamical systems is known in the literature with several different
names: \textit{non-stationary dynamical systems}, \textit{non-autonomous dynamical systems}, \textit{sequences of mappings}, among other names (see  \cite{alb}, \cite{Acastro}, \cite{Kawan}, \cite{KawanL}, \cite{Jeo3}, \cite{Jeo1} and the references there in).

\medskip

 Let $X_{1}$ and $X_{2}$ be $n$-dimensional compact Riemannian manifolds. We will introduce the metric on the spaces  $$\text{Diff}^{m}(X_{1},X_{2})=\{ \phi: X_{1}\rightarrow X_{2}: \phi \text{ is a } C^{m}\text{-diffeomorphism}\}\quad \text{ for }m=0,1,2.$$    Let $\Vert \cdot\Vert_{k}$   be the Riemannian norm on $X_{k}$ and $h_{k}: X_{k}\rightarrow \mathbb{R}^{p} $ be a $C^{\infty}$ isometric embedding, for $k=1,2$ and $p\in\mathbb{N}$ large enough.   
  We choose  a system of exponential charts $\{  (U_{x_{i}},\text{exp}_{x_{i}})\}_{i=1}^{l_{1}} $ and $\{  (V_{y_{i}},\text{exp}_{y_{i}})\}_{i=1}^{l_{2}}$ which covers $X_{1}$ and $X_{2}$ respectively, where $U_{x_{i}}$ is an open subset of $T_{x_{i}}X_{1}$ and $V_{y_{i}}$ is an open subset of  $T_{y_{i}}X_{2}$ with diameter less than $\varrho/2$, where $\varrho$ is an injectivity radius of $X_{2}$. We can identify isometrically $U_{x_{i}}$ and $V_{y_{i}}$  with open subsets of $\mathbb{R}^{n}$ and, by an abuse on the notation, we will continue calling them by $U_{x_{i}}$ and $V_{y_{i}}$. 
  
Let $d_{k}(\cdot,\cdot)$ be the metric   induced by $\Vert \cdot\Vert_{k}$ on $X_{k}$, for $k=1,2.$ Consider    two  homeomorphisms $\phi: X_{1}\rightarrow X_{2}$ and $\psi: X_{1}\rightarrow X_{2}$. 
 The $d^{0}$ metric on  $\text{Diff}^{0}(X_{1},X_{2})$ is given by
\begin{equation}\label{edssc1} d^{0}(\phi, \psi)=\max\left\{\max_{x\in X_{1}}d_{2}(\phi(x),\psi(x)),\max_{y\in X_{2}}d_{1}(\phi^{-1}(y),\psi^{-1}(y))\right\}.\end{equation}
 If $\varphi\in \text{Diff}^{0}(X_{1},X_{2})$,   take 
\[ \tilde{\varphi}_{x_{i}}= h_{2} \circ \varphi \circ \text{exp}_{x_{i}}  \text{ for }i=1,\dots,l_{1} \quad\text{and}\quad  \hat{\varphi}_{y}= h_{1} \circ \varphi^{-1} \circ \text{exp}_{y_{i}}   \text{ for } i=1,\dots,l_{2}.  \]
 The $d^{1}$ metric on  $\text{Diff}^{1}(X_{1},X_{2})$ is given by 
\begin{equation}\label{edssc2} d^{1}(\phi, \psi)=d^{0}(\phi, \psi) +\max_{1\leq i\leq l_{1}}\sup_{u\in U_{x_{i}}}\Vert D_{u}(\tilde{\phi}_{x_{i}}) - D_{u}(\tilde{\psi}_{x_{i}})\Vert+\max_{1\leq i\leq l_{2}}\sup_{v\in V_{y_{i}}}\Vert D_{v}(\hat{\phi}_{y_{i}}) - D_{v}(\hat{\psi}_{y_{i}})\Vert,\end{equation}
where $D_{w}\varphi$ is  the derivative of any diffeomorphism  $\varphi$  at $w\in \mathbb{R}^{n}$. 
If $\phi$ and $\psi$ are   $C^{2}$-diffeomorphisms, the $d^{2}$ metric   is given by 
\begin{equation}\label{edssc3}  d^{2}(\phi, \psi)=d^{1}(\phi, \psi)  +\max_{1\leq i\leq l_{1}}\sup_{u\in U_{x_{i}}}\Vert D_{u}^{2}(\tilde{\phi}_{x_{i}}) - D_{u}^{2}(\tilde{\psi}_{x_{i}})\Vert+\max_{1\leq i\leq l_{2}}\sup_{v\in V_{y_{i}}}\Vert D_{v}^{2}(\hat{\phi}_{y_{i}}) - D_{v}^{2}(\hat{\psi}_{y_{i}})\Vert,\end{equation}
where $D_{w}^{2}\varphi $ is  the second derivative of any $C^{2}$-map $\varphi$  at $w\in \mathbb{R}^{n}$.

\medskip

Fix $m\geq 1$. The set   \[\mathcal{D}^{m}(\mathcal{M})=\{\mathcal{F}=(f_{i})_{i\in\mathbb{Z}}: f_{i}:M_{i}\rightarrow M_{i+1} \text{ is a  }C^{m}\text{-diffeomorphism}\}\] 
 can be endowed with the \textit{strong topology} (or \textit{Whitney topology}) or the \textit{uniform topology}:

 \begin{defin}\label{weds}
 Given  $\mathcal{F}=(f_{i})_{i\in\mathbb{Z}}$ and  $\mathcal{G}=(g_{i})_{i\in\mathbb{Z}}$ in  $ \mathcal{D}^{m}(\mathcal{M})$, take 
 \[ d_{unif}^{m}(\mathcal{F},\mathcal{G})=\sup_{i\in\mathbb{Z}} \{\min \{d^{m}(f_{i},g_{i}),1)\}\},\]
 where $ d ^{m}(\cdot,\cdot)$ is the $C^{m}$-metric on $\text{Diff}^{m}(M_{i},M_{i+1})$ (see \eqref{edssc1}, \eqref{edssc2} and \eqref{edssc3}).  The \textit{uniform topology} on $ \mathcal{D}^{m}(\mathcal{M})$ is induced by $ d_{unif}^{m}(\cdot,\cdot)$. We denote by $\tau_{unif}$ the uniform topology on $ \mathcal{D}^{m}(\mathcal{M})$. 
 \end{defin}

 \begin{defin} For each   $\mathcal{F}=(f_{i})_{i\in\mathbb{Z}}\in  \mathcal{D}^{m}(\mathcal{M})$ and a sequence of positive numbers   $ (\epsilon_{i})_{i\in \mathbb{Z}}$, a  \textit{strong basic neighborhood} of $\mathcal{F}$ is the set $$B^{m}(\mathcal{F}, (\epsilon_{i})_{i\in \mathbb{Z}})=  \{\mathcal{G}\in \mathcal{D}^{m}(\mathcal{M}): d^{m}(f_{i},g_{i})<\epsilon_{i}\text{ for all }i\}.$$ 
  The $C^{m}$-\textit{strong topology}   is generated by the strong basic neighborhoods  of each $\mathcal{F}\in \mathcal{D}^{m}(\mathcal{M})$.\end{defin}

  A sequence   $(h_{i})_{i\in\mathbb{Z}}$, where $h_{i}:M_{i}\rightarrow M_{i}$ is a continuous maps for any $i\in\mathbb{Z}$, is equicontinuous if for any $\varepsilon>0$ there exists $\delta>0$ such that for any $i\in\mathbb{Z}$, if $x,y\in M_{i}$ with $d_{i}(x,y)<\delta$ then $d_{i}(h_{i}(x),h_{i}(y))<\epsilon$.  Non-stationary dynamical systems are classified via  \textit{uniform topological conjugacy}:

\begin{defin}\label{definconjugacy} A \textit{uniform topological conjugacy} between  two non-stationary systems $(\mathcal{M},\langle\cdot,\cdot\rangle,\mathcal{F})$  and  $(\mathcal{M},\langle\cdot,\cdot\rangle,\mathcal{G})$   
is a map $\mathcal{H}:\mathcal{M}\rightarrow \mathcal{M}$, such that, for  each $i\in \mathbb{Z} ,$ $\mathcal{H}|_{M_{i}}=h_{i}:M_{i}\rightarrow M_{i}$ is a homeomorphism, $(h_{i})_{i\in\mathbb{Z}}$  and $(h_{i}^{-1})_{i\in\mathbb{Z}}$ are equicontinuous and     \[h_{i+1}\circ f_{i}=g_{i}\circ h_{i}:M_{i}\rightarrow M_{i+1},\]   that is, the following diagram commutes: 
\[\begin{CD}
M_{-1}@>{f_{-1}}>> M_{0}@>{f_{0}}>>M_{1}@>{f_{1}}>>M_{2} \\
@V{\cdots}V{h_{-1}}V @VV{h_{0}}V @VV{h_{1}}V @VV{h_{2}\cdots}V\\
M_{-1}@>{g_{-1}}>> M_{0}@>{g_{0}}>>M_{1} @>{g_{1}}>>M_{2}
\end{CD} 
\]
In that case, we will say the families are \textit{uniformly conjugate}. 
\end{defin}

The reason for considering uniform conjugacy instead of topological conjugacy is that every non-stationary dynamical system is topologically conjugate to the constant family whose maps are all the
identity (see \cite{alb}, Proposition 2.1).

\medskip

The structural stability for elements in $\mathcal{D}^{m}(\mathcal{M})$ will be given considering the uniform topology on $ \mathcal{D}^{m}(\mathcal{M})$ (see Definition \ref{weds}).  

\begin{defin}\label{estruturalmenteestavel} We say that $\mathcal{F}\in \mathcal{D}^{m}(\mathcal{M})$ is \textit{uniformly structurally stable} if there exists $\varepsilon>0$ such that any  $\mathcal{G}\in \mathcal{D}^{m}(\mathcal{M})$,  with    $d_{unif}^{m}(\mathcal{F},
\mathcal{G})<\varepsilon$, is  uniformly  conjugate to $\mathcal{F}$. 
  A subset $ \mathcal{A}$ is called \textit{uniformly structurally stable} if all the elements in $ \mathcal{A}$ are uniformly structurally stable. 
\end{defin}

 Next,  the definition of Anosov families will be given.   It is important to keep fixed the Riemannian metric   on each $M_{i}$, since the notion of Anosov family depends on the Riemannian metric. The   hyperbolicity for the family could be induced by the Riemannian metrics $\langle \cdot, \cdot \rangle_{i}$ (see \cite{alb}, Example 4).

 \begin{defin}\label{anosovfamily}    An  \textit{Anosov family} on $\mathcal{M}$ is  a non-stationary dynamical system     $(\mathcal{M},\langle\cdot,\cdot\rangle, \mathcal{F})$ such that:
\begin{enumerate}[i.]
\item the tangent bundle $T\mathcal{M}$ has a continuous splitting   $E^{s}\oplus E^{u}$ which is  $D\mathcal{F}$-\textit{invariant}, i. e., for each $p\in \mathcal{M}$, 
 $T_{p}\mathcal{M}=E^{s}_{p}\oplus E^{u}_{p}$ with $D_{p}\mathcal{F}(E^{s}_{p})= E^{s}_{\mathcal{F}(p)}$ and $D_{p} \mathcal{F}(E^{u}_{p})= E^{u}_{\mathcal{F}(p)}$, where $T_{p}\mathcal{M}$ is the    tangent space at $p;$
\item there exist constants $\lambda \in (0,1)$ and $c>0$ such that for each  $i\in \mathbb{Z}$, $n\geq 1$,    and $p\in M_{i}$, 
we have: \[\Vert D_{p} (\mathcal{F}_{i}^{n})(v)\Vert \leq c\lambda^{n}\Vert v\Vert \text{ if   }v\in E_{p}^{s}\quad\text{and}\quad \Vert D_{p} (\mathcal{F}_{i}^{-n}) (v)\Vert \leq c\lambda^{n}\Vert v\Vert  \text{ if }v\in E_{p}^{u}.\]
\end{enumerate}
The  subspaces $E^{s}_{p}$ and $E^{u}_{p}$ are called     stable and unstable subspaces, respectively.
\end{defin}
  
If we can take
$c=1$ we say the family is \textit{strictly Anosov}. 
 We will denote by $\mathcal{A}^{1}(\mathcal{M})$ the set consisting of Anosov family on $\mathcal{M}$ such that each $f_{i}$ is a $C^{1}$-diffeomorphism.

 \medskip
 
 In a natural way, Fisher and Arnoux in \cite{alb} generalized the notion of Markov partition to non-
stationary dynamical systems, whose associated symbolic representation is a combinatorially defined
two-sided sequence of maps, which they called a \textit{nonstationary subshift of finite type}. This is a key
tool for the further study of Anosov families.
 
 \begin{obs}\label{exponential}  Proposition 3.5 in \cite{Jeo1}      proves     the notion of    Anosov family does not depend on   Riemannian metrics chosen  uniformly equivalent on the total space $\mathcal{M}$.
  On the other hand,  taking  a fixed Riemannian manifold $M$ and considering $M_{i}=M$ for each $i\in\mathbb{Z}$, we can  suitably change the Riemannian metric on each   $M_{i}$ for each $i\in\mathbb{Z}$ such that the sequence $(I_{i})_{i\in\mathbb{Z}}$, where $I_{i}:M \rightarrow M$ is the identity map, has a hyperbolic behavior (see \cite{alb}, Example 4).  
We want to exclude that kind of cases.  Therefore, we will suppose that \[  \varrho := \inf_{i\in\mathbb{Z}}\varrho_{i}>0,\] where $\varrho_{i}$ is an     injectivity radius  of  $M_{i}$.\footnote{In \cite{alb}, Example 4, we have $\varrho=0.$} Consequently,  for each  $i\in\mathbb{Z}$ and $p\in M_{i}$, the exponential map  \(\text{exp}_{p}:B (0_{p},\varrho) \rightarrow B ( p,\varrho) \) is a diffeomorphism and  
\begin{equation*}\label{exponentialmap} \Vert v\Vert =d_{i}(\text{exp}_{p}(v), p),\quad  \text{for all  }v\in B (0_{p},\varrho)
,\end{equation*}
where   $B (0_{p},\varrho)$ is the ball in $T_{p}M_{i}$ with radius   $\varrho$ and center $0_{p}\in T_{p}M_{i}$, the zero vector in $T_{p}M_{i}$,   and $B(p,\varrho)$ is the ball in $M_{i}$  with  radius $\varrho$ and center $p$.\end{obs}

 \begin{defin}\label{propang2} An Anosov family satisfies the \textit{property of the angles} (or \textit{s.p.a.}) if the angle between the stable and unstable subspaces are bounded away from zero. 
 \end{defin}

  \begin{obs}\label{remarnormas} In \cite{Jeo1},  Corollary 3.8, we proved that if $\mathcal{F}$ s.p.a. then  there exists a Riemannian norm      \( \Vert \cdot\Vert_{\ast}\),  
  uniformly equivalent to $\Vert \cdot\Vert$ on $T\mathcal{M}$, with which         $ \mathcal{F}$ is strictly Anosov, with constant $\tilde{\lambda}\in(\lambda,1)$. 
   \( \Vert \cdot\Vert_{\ast}\) is uniformly equivalent to the norm     given by
 \begin{equation}\label{metriceq} \Vert (v_{s},v_{u})\Vert_{\star}=\max\{\Vert v_{s}\Vert_{\ast} ,\Vert v_{u}\Vert _{\ast}\},   \text{ for }(v_{s},v_{u})\in E_{p}^{s}\oplus E_{p}^{u}, \text{ } p\in \mathcal{M}.
 \end{equation} 
 Consequently, there exists $C\geq 1$ such that 
  \begin{equation}\label{metriceqe} (1/C)\Vert v\Vert_{\star}\leq \Vert v \Vert \leq C \Vert v\Vert_{\star} ,   \text{ for every }v\in E_{p}^{s}\oplus E_{p}^{u}, \text{ } p\in \mathcal{M}.
 \end{equation} 
 \end{obs}
 
   From now on, $\Vert \cdot \Vert_{\star}$  will denote the norm given in \eqref{metriceq}.

 \medskip

In \cite{Jeo1} we proved that  $ \mathcal{A}^{1} (\mathcal{M})$  is an open subset of $ \mathcal{F}^{1} (\mathcal{M})$ with respect to the strong topology. This fact 
means that if $( f_{i} )_{i\in\mathbb{Z}}$ is an Anosov family, then there exists a two-sided sequence of positive numbers
 $(\epsilon_{i})_{i\in\mathbb{Z}}$ such that if $( g_{i} )_{i\in\mathbb{Z}}\in  \mathcal{D}^{1} (\mathcal{M})$ and $d^{1} ( f_{i} , g_{i} ) < \epsilon_{i}$  for any $i\in\mathbb{Z}$, then  $( g_{i} )_{i\in\mathbb{Z}}\in   \mathcal{A}^{1} (\mathcal{M})$. If we do not
ask for any additional condition on the family $( f_{i} )_{i\in\mathbb{Z}}$, the sequence $(\epsilon_{i})_{i\in\mathbb{Z}}$  could not be bounded away
from zero.

\section{Some examples of Anosov families}

 It is clear that if  for each $i\in\mathbb{Z}$  $f_{i}$ is an fixed Anosov diffeomorphism $\phi:M\rightarrow M$, where $M$ is a compact Riemannian manifold with Riemannian metric $\langle\cdot,\cdot\rangle$, then $(f_{i})_{i\in\mathbb{Z}}$ is an Anosov family, considering $M_{i}=M\times \{i\}$ endowed with the metric induced by $\langle\cdot,\cdot\rangle$. Furthermore, if $\phi$ is $C^{2}$ and each $f_{i}$ is a $C^{1+1}$ small perturbation of $\phi$, then $(f_{i})_{i\in\mathbb{Z}}$ is an Anosov family (see \cite{young}, Proposition 2.2). In this section we will show some results provide many examples of Anosov families  which are not necessarily sequences of   Anosov diffeomorphisms  (or small perturbations of a single Anosov diffeomorphism).

\begin{defin}\label{gathering} Let     $\mathcal{F}$ and  $\widetilde{\mathcal{F}}$  be  non-stationary dynamical systems  on $\mathcal{M}$ and $\widetilde{\mathcal{M}}$, respectively.  We say that  $ \widetilde{\mathcal{F}}$ is a \textit{gathering} of  $ \mathcal{F} $ if there exists a strictly increasing  sequence   of integers $(n_{i})_{i\in\mathbb{Z}}$ such that $\widetilde{M}_{i}=M_{n_{i}}$ and $\widetilde{\mathcal{F}}_{i}=f_{n_{i+1}-1}\circ \cdots \circ f_{n_{i}+1}\circ f_{n_{i}}$:  
\begin{equation*}\begin{CD}
\cdots M_{n_{i-1}}@>{\tilde{f}_{i-1}=f_{n_{i}-1}\circ \cdots  \circ f_{n_{i-1}}}>> M_{n_{i}}@>{\tilde{f}_{i}=f_{n_{i+1}-1}\circ \cdots  \circ f_{n_{i}}}>>M_{n_{i+1}}   \cdots
\end{CD} 
\end{equation*}
\end{defin}

It is not difficult to prove that:

\begin{propo}\label{propogathe}  Any gathering of an Anosov family   is also an Anosov family.  
\end{propo}

\begin{ej} It follows from Proposition \ref{propogathe} that if $\phi:M\rightarrow M$ is an Anosov diffeomorphism, then for each sequence of positive integers $(n_{i})_{i\in\mathbb{Z}}$, if $f_{i}=\phi^{n_{i}}$, then $(f_{i})_{i\in\mathbb{Z}}$ is an Anosov family.   \end{ej}

 In dimension 2, a necessary and sufficient condition for a family of matrices $\mathcal{A}=(A_{i})_{i\in\mathbb{Z}}$ acting on the $2$-torus $\mathbb{T}^{2}=\mathbb{R}^{2}/\mathbb{Z}^{2}$ by multiplication  on the
column vectors, is that  there exist constants $c > 0$ and $\sigma > 1$ such that
\[\Vert \mathcal{A}^{n} _{i}(x)\Vert\geq  c \sigma^{n} \quad\text{for all }x\in \mathbb{T}^{2} \text{ and }n \geq  1, i\in\mathbb{Z}\]
where $\Vert \cdot\Vert$ is the norm on $\mathbb{T}^{2}$ inherited from $\mathbb{R}^{2}$ (see \cite{Viana}, Proposition 2.1).

\medskip

 The following example, which is due to Arnoux and Fisher \cite{alb}, proves that Anosov families are not necessarily sequences of Anosov diffeomorphisms. 
 
\begin{ej}\label{familiamultiplicativa}   For any sequence of positive integers  $(n_{i})_{i\in \mathbb{Z}}$ set
 \begin{equation*} A_{i} =
\left(
\begin{array}{ccc}
1 & 0 \\
n_{i} & 1
 \end{array}
\right) \text{ for $i$ even}\quad\text{ and } \quad A_{i} =
\left(
\begin{array}{ccc}
1 & n_{i} \\
0 & 1
 \end{array}
\right) \text{ for $i$ odd},
\end{equation*}
acting on the 2-torus $M_{i}=\mathbb{T}^{2}$. 
The family $(A_{i})_{i\in \mathbb{Z}}$ is called  the  \textit{multiplicative family   determined by the sequence} $(n_{i})_{i\in \mathbb{Z}}$.   Let $\Vert\cdot \Vert$ be the Riemannian  metric on $M_{i}$     inherited from  $\mathbb{R}^{2}$.
For each $i\in \mathbb{Z}$, let  $s_{i} =(a_{i},b_{i})$, $ u_{i} =(c_{i},  d_{i})$ and $\lambda_{i}\in (0,1 )$ be such that  $ a_{i}d_{i} + c_{i}b_{i}=1, $
\[\text{for }i\text{ even, } a_{i}=[n_{i}n_{i+1}...], \text{ } b_{i}=1, \text{ } \frac{d_{i}}{c_{i}}=[n_{i-1}n_{i-2}...],  \text{ and } \lambda_{i}= a_{i},\]
and 
\[\text{for $i$ odd, } b_{i}=[n_{i}n_{i+1}...], \text{ } a_{i}=1, \text{ }\frac{c_{i}}{d_{i}}=[n_{i-1}n_{i-2}...] \text{ and } \lambda_{i}= b_{i}.\] 
Here,  $$[n_{i}n_{i+1}...]=\frac{1}{n_{i}+\frac{1}{n_{i+1}+\cdots}}.$$  
For all $i\in \mathbb{Z}$ and $n\geq1$,  we have  \begin{equation}\label{edfsrre}\Vert A_{i}^{n}s_{i}\Vert \leq c\lambda_{i+n-1} \cdots \lambda_{i}  \Vert s_{i}\Vert\quad \text{\quad and \quad} \Vert A_{i}^{n}u_{i}\Vert\geq c^{-1}\lambda_{i+n-1}^{-1}\cdots\lambda_{i}^{-1} \Vert u_{i}\Vert,\end{equation}
 where \(c=\max\left\{\sup_{i,j}\left\{\frac{\Vert s_{i}\Vert}{\Vert s_{j}\Vert}\right\},\sup_{i,j}\left\{\frac{\Vert u_{i}\Vert}{\Vert u_{j}\Vert}\right\}\right\}\) ($c<\infty$ because $\Vert v\Vert \in (1/2,\sqrt{2})$ for all $v\in \{s_{i}:i\in \mathbb{Z}\}\cup \{u_{i}:i\in \mathbb{Z}\} $). 
\end{ej}

Note that, if there exists $\lambda \in(0,1)$ such that $\lambda_{i}\leq\lambda$ for all $i$, we have   
$$\Vert A_{i}^{n}s_{i}\Vert\leq c\lambda ^{n}\Vert s_{i}\Vert\quad\text{ and }\quad \Vert A_{i}^{n}u_{i}\Vert\geq c^{-1}\lambda^{-n} \Vert u_{i}\Vert\text{ for all  }n\geq 1.$$
This shows that, if there is a $ \lambda\in(0, 1) $ such that $ \lambda_{i} \leq  \lambda $ for all $ i$, then $ (A_{i})_{i\in\mathbb{Z}} $ is an Anosov family, with constants $ \lambda  $ and $ c $ as defined above,   the stable subspaces are spanned by $ s_{i} $ and the unstable subspaces are spanned by $ u_{i} $.   However, we have:

\begin{propo}\label{multiplicativef2} Any multiplicative family is an Anosov family with constant $\lambda=\sqrt{2/3}$ and $2c$. \end{propo}
\begin{proof} Notice that, if $\lambda_{j}\in (2/3,1)$ for some $j\in\mathbb{Z}$, then $\lambda_{j-1}\in (0,2/3)$ and $\lambda_{j+1}\in (0,1/2).$ Indeed, if  $\lambda_{j}=\frac{1}{n_{j}+\frac{1}{n_{j+1}+\cdots}}\in  (2/3,1)$ we must have   $n_{j}=1$ and $n_{j+1}\geq 2$. Hence,   \[\lambda_{j-1}=\frac{1}{n_{j-1}+\frac{1}{1+\cdots}}  < \frac{1}{1+(1/2)}\quad \text{and}\quad \lambda_{j+1}=\frac{1}{n_{j+1}+\frac{1}{n_{j+2}+\cdots}}<1/2.\] 

Next, by induction on $n$, we  prove that $c  \lambda_{i+n-1}\cdots\lambda_{i} <2c\lambda ^{n}$, for each $i\in\mathbb{Z}$ and $n\geq 1$. Fix $i\in\mathbb{Z}$. It is clear that if $n=1,2,$ then $c  \lambda_{i+n-1}\cdots\lambda_{i} <2c\lambda ^{n}$. Let $n\geq 2$ and assume that $c  \lambda_{i+m-1}\cdots\lambda_{i} <2c\lambda ^{m}$ for each $m\in \{1,\dots,n\}$.  Clearly, if $\lambda_{i+n}\leq 2/3, $ then $c  \lambda_{i+n}\cdots\lambda_{i} <2c\lambda ^{n+1}.$ On the other hand, if $\lambda_{i+n}> 2/3$, then $\lambda_{i+n-1}<1/2$ and by induction assumption we have 
\[ c  \lambda_{i+n}\lambda_{i+n-1}\cdots\lambda_{i}  < (\lambda_{i+n}\lambda_{i+n-1}) 2c\lambda ^{n-1}  < \frac{1}{2}\cdot2c\lambda ^{n-1} <\lambda ^{2} 2c\lambda ^{n-1}=  2c\lambda ^{n+1}.\]

It follows from \eqref{edfsrre} and  the above facts   that $$\Vert A_{i}^{n}s_{i}\Vert\leq 2c\lambda^{n}  \Vert s_{i}\Vert\quad\text{and}\quad \Vert A_{i}^{n}u_{i}\Vert\geq (2c)^{-1}\lambda^{-n} \Vert u_{i}\Vert,$$ for each $i\in\mathbb{Z}$ and $n\geq 1$, which proves the proposition. \end{proof} 
 
By Proposition \ref{propogathe} we have any gathering of a multiplicative family is an Anosov family.  On the other hand, if $F_{i}\in SL(\mathbb{N},2)$, then \begin{equation}\label{recc}F_{i} =
\left(
\begin{array}{ccc}
1 & 0 \\
n_{i,k_{i}} & 1
 \end{array}
\right) \left(
\begin{array}{ccc}
1 & n_{i,k_{i}-1} \\
0 & 1
 \end{array}
\right)\cdots \left(
\begin{array}{ccc}
1 & 0 \\
n_{i,2} & 1
 \end{array}
\right)\left(
\begin{array}{ccc}
1 & n_{i,1} \\
0  & 1
 \end{array}
\right) ,
\end{equation}
for some non-negative integers $n_{i,1},\dots,n_{i,k_{i}}$, that is,  $SL(\mathbb{N},2)$ is a semigroup ge\-nerated by  \begin{equation*} M =
\left(
\begin{array}{ccc}
1 & 0 \\
1 & 1
 \end{array}
\right)  \quad\text{ and } \quad N =
\left(
\begin{array}{ccc}
1 & 1 \\
0 & 1
 \end{array}
\right) 
\end{equation*}
(see \cite{alb}, Lemma 3.11). 

\begin{cor}\label{wed} Consider a sequence $(F_{i})_{i\in\mathbb{Z}}$ in $SL(\mathbb{N},2)$ and the factorization of each $F_{i}$ as in \eqref{recc}. If $n_{i,k_{i}}$ and $n_{i,1}$ are non-zero for each $i\in\mathbb{Z}$, then $(F_{i})_{i\in\mathbb{Z}}$ is an Anosov family. \end{cor} 
\begin{proof} Notice that $(F_{i})_{i\in\mathbb{Z}}$ is a gathering of an multiplicative family.   It follows from Proposition \ref{multiplicativef2} that $(F_{i})_{i\in\mathbb{Z}}$ is an Anosov family. \end{proof}

 Corollary \ref{wed} provides a great variety of examples of Anosov families. Next, suppose that $X=\{B_{1},\dots,B_{k}\}\subseteq  SL(\mathbb{N},2)$.  If each $F_{i}\in X$, then  $(F_{i})_{i\in\mathbb{Z}}$ is an Anosov family   (see \cite{Viana}, Proposition 2.7). Another examples are provided by Theorem 5.1 in \cite{alb}: if $(A_{i})_{i\in \mathbb{Z}}\subseteq  SL(\mathbb{N},2)$ is a non-eventually constant sequence of matrices with non-negative entries, then it is an Anosov family on $\mathbb{T}^{2}$.

\section{Characterization of Anosov families}

Let $M$ be a compact Riemannian manifold and denote by $ \Gamma(M)$ the set consisting of continuous
sections of $T M$. For a diffeomorphism $f$ on $M$, set $f_{\ast} : \Gamma(M) \rightarrow  \Gamma(M) $ the bounded linear operator
defined by $f_{\ast}(\zeta) = D f \circ \zeta \circ f^{-1}$ for any $\zeta \in \Gamma(M)$. J. Mather in \cite{Mather1}, \cite{Mather12} proved that $ f$ is an Anosov
diffeomorphism if and only if $f_{\ast}$ is a hyperbolic automorphism. Furthermore, he proved that this is
equivalent to show that $f_{\ast}- I$ is an automorphism on $ \Gamma(M)$, where $I$  is the identity on $ \Gamma(M)$. We
can define the operator $f_{\ast}$ for non-stationary dynamical systems. In this section we will give some
conditions on an Anosov family to obtain the hyperbolicity of $f_{\ast}$ defined for such family.

 %\medskip
 
% The followings are some notations to be used from now on.

\begin{defin}\label{subestp}   For $\tau>0$ and $i\in\mathbb{Z}$, set:
\begin{enumerate}[\upshape (i)]
\item $D(I_{i},\tau) = \{h:M_{i}\rightarrow M_{i}: h \text{ is } C^{0} \text{ and } d(h(p), I_{i}(p))\leq \tau \text{ for any } p\in M_{i} \};$
\item $ \mathcal{D} (\tau) = \{(h_{i})_{i\in\mathbb{Z}}: h_{i}\in  D(I_{i},\tau) \text{ for any } i\in\mathbb{Z}    \}$;
\item $\Gamma(M_{i})=\{\zeta:M_{i}\rightarrow TM_{i}: \sigma \text{ is a continuous section} \};$
\item $\Gamma_{\tau}(M_{i})=\{\zeta \in \Gamma(M_{i}) : \sup_{p\in M_{i}}\Vert \zeta(p)\Vert \leq \tau\};$
\item \( \Gamma (\mathcal{M}) = \left\{ (\zeta _{i})_{i\in\mathbb{Z}} : \zeta_{i}\in \Gamma  (M_{i}) \text{ and }\sup_{i\in\mathbb{Z}}\Vert \zeta_{i}\Vert _{\Gamma_{i}}<\infty\right\},\) where   \( \Vert
\zeta\Vert _{\Gamma_{i}}= \max_{p\in M_{i}} \Vert \zeta (p)\Vert.\) 
\item \( \Gamma_{\tau} (\mathcal{M}) = \left\{ (\zeta _{i})_{i\in\mathbb{Z}}\in \Gamma (\mathcal{M}) : \zeta_{i}\in \Gamma_{\tau}  (M_{i}) \text{ for each } i\in\mathbb{Z} \right\};\) 
\end{enumerate}  
 \end{defin}
 
It is clear that $\Gamma(\mathcal{M})$ is a proper subset of 
$\prod_{i=-\infty}^{\infty}\Gamma(M_{i})$. Note that   $\Gamma  (M_{i})$ is a Banach  space with the norm \( \Vert
\cdot\Vert _{\Gamma_{i}} .\) 
Therefore: 

\begin{lem}     $\Gamma(\mathcal{M})$ 
is a Banach space with the norm 
\( \Vert (\sigma _{i})_{i\in\mathbb{Z}} \Vert _{\infty}=  \sup_{i\in\mathbb{Z}}\Vert \sigma_{i}\Vert _{\Gamma_{i}}.\)  
\end{lem}  
  
  \begin{defin}\label{poerator} For any n.s.d.s. $ \mathcal{F}=(f_{i})_{i\in\mathbb{Z}}\in \mathcal{D}^{1}(\mathcal{M})$, define  
\begin{align*} \textbf{F}:\mathcal{D}(\textbf{F})&\rightarrow \Gamma(\mathcal{M})\\
(\zeta_{i})_{i\in\mathbb{Z}}&\mapsto (\textbf{F}_{i-1}(\zeta_{i-1}))_{i\in \mathbb{Z}}
\end{align*}  
where   $\textbf{F}_{i}:\Gamma(M_{i})\rightarrow \Gamma(M_{i+1})$ is defined by the formula
\(\textbf{F}_{i}(\zeta)(p) = D _{f_{i}^{-1}(p)}(f_{i}) (\zeta (f_{i}^{-1}(p))),\)  for \(p\in M_{i+1},  \zeta\in \Gamma (M_{i}),\) and \[ \mathcal{D}(\textbf{F})=\{ \zeta \in\Gamma(\mathcal{M}):\textbf{F}(\zeta)\in\Gamma(\mathcal{M})\}.\]
\end{defin}

It is not difficult to prove that $ \textbf{F}$ is a linear operator and    
\begin{equation}\label{eqqq} \Vert \textbf{F}\Vert=\sup_{\Vert(\zeta_{i})\Vert_{\infty}=1}\Vert  \textbf{F}_{i}(\zeta_{i})\Vert_{\infty}\leq \sup_{i\in\mathbb{Z}}\Vert Df_{i}\Vert.\end{equation}
Therefore, if $\sup_{i\in\mathbb{Z}}\Vert Df_{i}\Vert<\infty,$ then  $\textbf{F}$ is a bounded linear operator, and, in this case, $$\mathcal{D}(\textbf{F})=\Gamma(\mathcal{M}).$$  However, in general, we do  not have $\sup_{i\in\mathbb{Z}}\Vert Df_{i}\Vert<\infty$ (see Example \ref{familiamultiplicativa}). Consequently, \textbf{F} could be an unbounded operator. 

\medskip

In this section  we will suppose that that $\mathcal{F}\in \mathcal{A}^{1}(\mathcal{M})$ with constants $\lambda\in (0,1)$ and $c\geq1.$ Consider the splitting $TM_{i}=E_{i}^{s}\oplus E_{i}^{u}$ induced by $ \mathcal{F}$ for each $i\in \mathbb{Z}$. Set 
  $$\Gamma^{s}(M_{i})=\{\sigma\in \Gamma(M_{i}): \text{ the image of }\sigma \text{ is contained in } E^{s} \}$$
  and
  $$\Gamma^{u}(M_{i})=\{\sigma\in \Gamma(M_{i}): \text{ the image of }\sigma \text{ is contained in } E^{u} \}.$$

It is clear that \[\Gamma(M_{i})=\Gamma^{s}(M_{i})\oplus \Gamma^{u}(M_{i}),\quad \textbf{F}_{i} (\Gamma^{s}(M_{i}))=\Gamma^{s}(M_{i+1}) \quad\text{ and }\quad \textbf{F}_{i}(\Gamma^{u}(M_{i}))=\Gamma^{u}(M_{i+1}).\]

For $t=s,u$, set \begin{equation}\label{dff} \Gamma ^{t}(\mathcal{M})=\{ (\sigma_{i})_{i\in\mathbb{Z}}\in \Gamma (\mathcal{M}): \sigma_{i} \in \Gamma^{t}(M_{i}) \text{ for each }i\in\mathbb{Z} \}.\end{equation} 

Note that $\Gamma^{s}(\mathcal{M})$ and $\Gamma^{u}(\mathcal{M})$ are closed subspaces of $\Gamma(\mathcal{M})$ and $ \Gamma^{s}(\mathcal{M})\cap \Gamma^{u}(\mathcal{M})=\{\tilde{0}\}$, the zero vector in $ \Gamma(\mathcal{M}).$

\begin{lem}\label{lema44}  $ \Gamma^{s}(\mathcal{M})$ and  $ \Gamma^{u}(\mathcal{M})$ are complementary  subspaces of $ \Gamma (\mathcal{M})$ (that is,  \( \Gamma (\mathcal{M}) = \Gamma ^{s}(\mathcal{M})  \oplus \Gamma ^{u}(\mathcal{M})  )\) if and only if $\mathcal{F}$ satisfies the property of angles. 
\end{lem}
\begin{proof} 
Suppose that $\mathcal{F}$ s.p.a. Let us prove that $ \Gamma(\mathcal{M}) = \Gamma^{s}(\mathcal{M})\oplus \Gamma^{u}(\mathcal{M})$. For $\zeta= (\zeta_{i})_{i\in\mathbb{Z}}\in \Gamma(\mathcal{M})$, take
\begin{equation}\label{dfdf} \zeta_{i}^{s}(p) = \pi^{s} (p)\zeta_{i}(p)\quad\text{
and}\quad 
\zeta_{i}^{u} (p) = \pi^{u}(p)\zeta_{i}(p), \quad\text{
for }i \in\mathbb{Z}, p \in M_{i} ,\end{equation}
where $\pi^{s}(p)$ and $\pi^{u}(p)$  are the projections on $E^{s}_{p}$  and $E_{p}^{u}$, respectively. Since $\mathcal{F}$ satisfies the property of
the angles, there exists a $K > 0 $ such that
\[\Vert \zeta_{i}^{s}(p)\Vert \leq K\Vert \zeta \Vert_{\infty}
\quad\text{and}\quad \Vert \zeta_{i}^{u}(p)\Vert \leq K\Vert \zeta \Vert_{\infty}, \text{ for } i \in\mathbb{Z}, p \in  M_{i} .\]
Therefore $\zeta^{s} = (\zeta^{s}_{i} )_{i\in \mathbb{Z}} \in \Gamma^{s}(\mathcal{M}), \zeta^{u} = (\zeta^{u}_{i} )_{i\in \mathbb{Z}}  \in \Gamma^{u}(\mathcal{M})$  and $\zeta = \zeta^{s} + \zeta^{u}$, which proves that $ \Gamma(\mathcal{M}) = \Gamma^{s}(\mathcal{M})\oplus \Gamma^{u}(\mathcal{M})$. 

Now, suppose that $\mathcal{F}$  does not satisfies the property of angles. Without loss of generality, we can
assume there exist $p \in M_{0}$  and a subsequence $k_{1} < k_{2} < k_{3} < . . . $ in $\mathbb{N}$ such that the angle $\theta_{k_{i}}$  between
the subspaces $   E_{\mathcal{F}_{0}^{k_{i}}(p)}^{s}$ and $   E_{\mathcal{F}_{0}^{k_{i}}(p)}^{u}$  converges to zero as $
i\rightarrow + \infty$. We can suppose that the angle between $ E_{\mathcal{F}_{0}^{k_{i}}(p)}^{s}$ and $   E_{\mathcal{F}_{0}^{k_{i}}(p)}^{u}$ is less than $\pi/4$ and $\theta_{k_{i+1}}<\theta_{k_{i}}$ for every $i\geq 1$.  For each $i\geq 1$,  take $v_{k_{i}}^{s}\in   E_{\mathcal{F}_{0}^{k_{i}}(p)}^{s}$ and $v_{k_{i}}^{u}\in    E_{\mathcal{F}_{0}^{k_{i}}(p)}^{u}$ such that for all $i\geq  1,$ 
$$    \widehat{v_{k_{i}}^{s} v_{k_{i}}^{u}}=\pi -\theta_{k_{i}} ,\quad \frac{1}{\sin(\theta_{k_{i-1}})}<\Vert v_{k_{i}} ^{u} \Vert <\frac{1}{\sin(\theta_{k_{i}})},   \quad \text{and} \quad \Vert v_{k_{i}}^{s}\Vert  =\sqrt{1-\sin^{2}(\theta_{k_{i}})\Vert v_{k_{i}}^{u}\Vert^{2}}-\cos(\theta_{k_{i}})\Vert  v_{k_{i}}^{u}\Vert,$$   where $ \widehat{v_{k_{i}}^{s} v_{k_{i}}^{u}}$ is the angle between $v_{k_{i}}^{s}$ and $v_{k_{i}}^{u}$.
 Since  
\begin{equation*}  \Vert v_{k_{i}}^{s}+ v_{k_{i}}^{u}\Vert^{2} =\Vert v_{k_{i}}^{s}\Vert^{2} + \Vert v_{k_{i}}^{u}\Vert^{2}  + 2\cos(\theta_{k_{i}})\Vert v_{k_{i}}^{s}\Vert \Vert v_{k_{i}}^{u}\Vert \quad \text{for each } i\geq 1,\end{equation*}
we can prove that $$ \Vert v_{k_{i}}^{s}+ v_{k_{i}}^{u}\Vert=1 \quad  \text{for each } i\geq 1.$$
However, $\Vert v_{k_{i}}^{u}\Vert \rightarrow \infty$  as $i \rightarrow \infty$ and, therefore, $\Vert v_{k_{i}}^{s}\Vert \rightarrow \infty$  as $i \rightarrow \infty$. Take $\zeta = (\zeta_{i})_{i\in\mathbb{Z}}$ such that $\zeta_{k_{i}}(\mathcal{F}^{k_{i}}(p)) = v_{k_{i}}^{s} + v_{k_{i}}^{u}$ and $\Vert \zeta_{i}\Vert_{\Gamma_{i}}\leq 1$ for any $i\in\mathbb{Z}$. Note that the sections $\zeta^{s} = (\zeta_{i}^{s}) _{i\in\mathbb{Z}}$ and $\zeta^{u} = (\zeta_{i}^{u}) _{i\in\mathbb{Z}}$
defined in \eqref{dfdf}  are the only that satisfy $\zeta_{i}^{s}\in \Gamma^{s}(M_{i}),$  $\zeta_{i}^{u}\in \Gamma^{u}(M_{i})$ and $\zeta=\zeta^{s}+\zeta^{u}.$    However $\zeta_{i}^{s}$ is not bounded and therefore does not belong to $\Gamma^{s}(\mathcal{M}).$ Hence $\zeta\in \Gamma(\mathcal{M})\setminus \Gamma^{s}(\mathcal{M})\oplus \Gamma^{u}(\mathcal{M})$, that is, $\Gamma^{s}(\mathcal{M})$ and  $\Gamma^{u}(\mathcal{M})$ are not complementary in  $\Gamma(\mathcal{M})$ 
\end{proof}

Every $C^{1}$-Anosov diffeomorphism $\phi : M \rightarrow M $ defined on a compact Riemannian manifold $M$
satisfies the property of angles, because the compactness of $M$. Hence, the subspaces $\Gamma^{s}(M)$ and  $\Gamma^{u}(M)$ defined in \eqref{dff} are complementary subspaces in  $\Gamma(M)$. Example 2.3 in \cite{Jeo2} proves that there exist Anosov families which do
not satisfy the property of angles. In that case, for any $p \in  M_{0}$, the angle between $E_{\mathcal{F}^{n} (p)}^{s}$ and $E_{\mathcal{F}^{n} (p)}^{u}$ 
converges to zero as $n \rightarrow\pm \infty$.

\medskip

Now, following the Mather's ideas in \cite{Mather1} and \cite{Mather12} we obtain the next characterization for Anosov
families which s.p.a. and with bounded derivative.

 \begin{teo}\label{characte}  $\textbf{F}:\Gamma(\mathcal{M})\rightarrow \Gamma(\mathcal{M})$ is a bounded hyperbolic automorphism if and only if $\mathcal{F}\in\mathcal{A}^{1}(\mathcal{M})$ s.p.a. and $\sup_{i\in\mathbb{Z}} \Vert D f_{i}\Vert <\infty$. 
 \end{teo}
 
 \begin{proof}
 Suppose that $ \mathcal{F} \in  \mathcal{A}^{1} (\mathcal{M})$  s.p.a. and $\sup_{i\in\mathbb{Z}} \Vert D f_{i}\Vert <\infty$. Thus \textbf{F} is a bounded automorphism on
$\Gamma(\mathcal{M})$. If $n\geq 1$, then 
\begin{equation}\label{dccd} \Vert\textbf{F}_{i+n-1}\circ \cdots \circ \textbf{F}_{i} (\zeta)\Vert _{\Gamma_{i+n}}\leq c{\lambda}^{n} \Vert \zeta \Vert _{\Gamma_{i}}     \text{ if }  \zeta\in \Gamma^{s}(M_{i})\end{equation}
 and \begin{equation}\label{eddee}\Vert\textbf{F}_{i+n-1}\circ \cdots \circ\textbf{F}_{i} (\zeta)\Vert _{\Gamma_{i+n}}\geq c^{-1}{\lambda} ^{-n}\Vert \zeta\Vert _{\Gamma_{i}} \text{ if }  \zeta\in \Gamma^{u}(M_{i}).\end{equation}
 
 It follows from Lemma \ref{lema44} that $ \Gamma^{s}(\mathcal{M})$ and $ \Gamma^{u}(\mathcal{M})$, defined in \eqref{dff}, are complementary subspaces
of $ \Gamma(\mathcal{M})$. Furthermore, we have they are invariant by $\textbf{F}$. Therefore,  $$\sigma(\textbf{F}) = \sigma(\textbf{F}|_{ \Gamma^{s}(\mathcal{M})})\cup \sigma(\textbf{F}|_{ \Gamma^{u}(\mathcal{M})})$$ (see
\cite{TKato}). It follows from \eqref{dccd} and \eqref{eddee} that for any $n\geq 1$ 
\[ \Vert\textbf{F}^{n}(\zeta)\Vert _{\infty}\leq c{\lambda}^{n} \Vert \zeta \Vert _{\infty}     \text{ for }  \zeta\in \Gamma^{s}(\mathcal{M})\quad\text{ and }\quad  \Vert\textbf{F}^{n}(\zeta)\Vert _{\infty}\geq c^{-1}{\lambda}^{-n} \Vert \zeta \Vert _{\infty} \text{ for }  \zeta\in \Gamma^{u}(\mathcal{M}).\]

Hence \[ \lim_{n\rightarrow \infty} \Vert (\textbf{F}|_{\Gamma^{s}(\mathcal{M})})^{n}\Vert ^{1/n}\leq \lambda <1 \quad \text{and} \quad  \lim_{n\rightarrow \infty} \Vert (\textbf{F}|_{\Gamma^{u}(\mathcal{M})})^{-n}\Vert ^{1/n}\geq \lambda^{-1}>1  .\]
This fact proves that $$\sigma(\textbf{F}|_{\Gamma^{s}(\mathcal{M})})
\subseteq \{z \in\mathbb{C} : \Vert z\Vert\leq \lambda\} \quad\text{and}\quad \sigma(\textbf{F}|_{\Gamma^{u}(\mathcal{M})})
\subseteq \{z \in\mathbb{C} : \Vert z\Vert\geq \lambda^{-1}\}$$
(see \cite{TKato}). Hence $\textbf{F}$ is a hyperbolic automorphism. 

Now, assume that $\textbf{F}:\Gamma(\mathcal{M})\rightarrow \Gamma (\mathcal{M})$ is  a bounded hyperbolic automorphism. Thus, there exist two closed subspaces $\Gamma^{1}$ and $\Gamma^{2}$ of $\Gamma(\mathcal{M})$ such that 
\begin{enumerate}[i.]
    \item $\Gamma(\mathcal{M})=\Gamma^{1}\oplus \Gamma^{2}$;
    \item $\textbf{F}(\Gamma^{1})=\Gamma^{1}$ and $\textbf{F}(\Gamma^{2})=\Gamma^{2}$;
    \item there exist $\lambda_{1}$, $\lambda_{2}\in (0,1)$ such that 
    \begin{equation}\label{2wew2}
        \lim_{n\rightarrow \infty} \Vert (\textbf{F}|_{\Gamma^{1}})^{n}\Vert ^{1/n}=\lambda_{1}\quad \text{and} \quad  \lim_{n\rightarrow \infty} \Vert (\textbf{F}|_{\Gamma^{2}})^{-n}\Vert ^{1/n}=\lambda_{2};  \end{equation}
\item $\sigma(\textbf{F})=\sigma(\textbf{F}|_{\Gamma^{1}})\cup \sigma(\textbf{F}|_{\Gamma^{2}})$, with $\sigma(\textbf{F}|_{\Gamma^{1}} )\subseteq \{ z\in\mathbb{Z}:0< \vert z\vert \leq \lambda_{1}\} $ and  $\sigma(\textbf{F}|_{\Gamma^{2}}) \subseteq \{ z\in\mathbb{Z}:    \lambda_{2}^{-1} \leq \vert z\vert\} $ (see \cite{TKato}). \end{enumerate}

For each $j\in\mathbb{Z}$, set 
\[ \Gamma_{j}^{1}=\{ \zeta=(\zeta_{i})_{i\in\mathbb{Z}} \in\Gamma(\mathcal{M}):\zeta_{i}=0\text{ for } i\neq j \text{ and } \limsup_{n\rightarrow \infty} \Vert \textbf{F}^{n}(\zeta)\Vert ^{1/n} \leq \lambda_{1}\} \] and 
\[ \Gamma_{j}^{2}=\{ \zeta=(\zeta_{i})_{i\in\mathbb{Z}} \in\Gamma(\mathcal{M}):\zeta_{i}=0\text{ for } i\neq j \text{ and } \limsup_{n\rightarrow \infty} \Vert \textbf{F}^{-n}(\zeta)\Vert ^{1/n} \leq \lambda_{2}\} .\]
Thus $\Gamma_{j}^{1}$ and $\Gamma_{j}^{2}$ are subspaces of $\Gamma(\mathcal{M})$ for each $j\in\mathbb{Z}$. We can prove that 
\begin{equation}\label{desdeee} \textbf{F}(\Gamma^{1}_{j})=\Gamma^{1}_{j+1}\text{ and }\quad \textbf{F}(\Gamma^{2}_{j})=\Gamma^{2}_{j+1}\quad \text{ for each }j\in\mathbb{Z}.\end{equation} 
Next, the projection 
\begin{align*}\Pi_{j}: \Gamma(\mathcal{M})&\rightarrow \Gamma(M_{j})\\
(\zeta_{i})_{i\in\mathbb{Z}}&\mapsto \zeta_{j}
\end{align*}
induces the isomorphism 
\begin{align*}\Pi_{j}|_{\Gamma^{r}_{j}}: \Gamma^{r}_{j}& \rightarrow \Pi_{j}(\Gamma^{r}_{j})\subseteq \Gamma(M_{j})\\
(\zeta_{i})_{i\in\mathbb{Z}}&\mapsto \zeta_{j}
\end{align*}
for $r=1,2$ and $j\in\mathbb{Z}$. Since $\Gamma(\mathcal{M})=\Gamma^{1}\oplus \Gamma^{2},$ we have 
\begin{equation}\label{defrr}
    \Gamma(M_{j})=\Pi_{j}(\Gamma^{1}_{j})=\Pi_{j}(\Gamma^{1}_{j})\oplus \Pi_{j}(\Gamma^{2}_{j}).
\end{equation}
From \eqref{desdeee} we can prove that 
\begin{equation}\label{edfffr}
    Df_{j}(\Pi_{j}(\Gamma_{j}^{r}))=\Pi_{j+1}(\Gamma_{j+1}^{r}) \quad \text{ for } j\in\mathbb{Z} \quad \text{ and } r=1,2.
\end{equation}
By the definition of $\Pi_{j}( \Gamma_{j}^{r})$, we obtain \begin{align*} \Pi_{j} (\Gamma_{j}^{1})  &= \{ \zeta\in \Gamma(M_{i}): \limsup_{n\rightarrow \infty} \Vert D\textbf{\textit{f}}^{n}(\zeta)\Vert_{\Gamma_{j+n}} ^{1/n}\leq \lambda_{1}\}\\
\Pi_{j} (\Gamma_{j}^{2})  &= \{ \zeta\in \Gamma(M_{i}): \limsup_{n\rightarrow \infty} \Vert D\textbf{\textit{f}}^{-n}(\zeta)\Vert_{\Gamma_{j-n}} ^{1/n}\leq \lambda_{2}\}.\end{align*}
It is not difficult to prove that $\Pi_j (\Gamma^{r}_{j})$ is a $C^{0}(M_{j},\mathbb{R})$ submodule of $\Gamma(M_{j})$ for $r=1,2,$ where 
\[C^{0}(M_{j},\mathbb{R})= \{ \phi : M_{j}\rightarrow \mathbb{R}: \phi \text{ is a continuous map}\} .\]
Therefore, it follows from \cite{Swan}, Theorem 2, that $ \Gamma(M_{j})$ is isomorphic to $TM_{j}$ and $\Gamma_{j}^{1}$ and  $\Gamma_{j}^{2}$ are, respectively, isomorphic to the subbundles $E_{j}^{1}$ and $E_{j}^{2}$ of $TM_{j}$ for each $j\in\mathbb{Z}$. Taking $\lambda =\max\{ \lambda_{1}, \lambda_{2}\},$ it follows from \eqref{2wew2} that there exists $c>0$ such that 
\begin{equation}\label{dcded}
   \Vert D_{p}(\mathcal{F}_{j}^{n})(v)\Vert\leq c \lambda^{n}\Vert v\Vert \text{ if }v\in (E_{j}^{1})_{p} \quad\text{and}\quad   \Vert D_{p}(\mathcal{F}_{j}^{-n})(v)\Vert\leq c \lambda^{n}\Vert v\Vert \text{ if }v\in (E_{j}^{2})_{p}.
\end{equation}
We have from \eqref{defrr}, \eqref{edfffr} and \eqref{dcded} that $\mathcal{F}$ is an Anosov family with constants $\lambda\in (0,1)$ and $c>0$ (the continuity of each $(E_{j}^{r})_{p}$ on $p$ if follows from Proposition 3.4 in \cite{Jeo1}). 

Next, we will prove that $\mathcal{F}$
satisfies the property of angles. It is clear that $\Gamma^{s} (\mathcal{M})\subseteq \Gamma^{1}$  and $\Gamma^{u} (\mathcal{M}) \subseteq \Gamma^{2}$. Let us show that $\Gamma^{1}\subseteq \Gamma^{s}(\mathcal{M})$. Suppose that there exists $\zeta=(\zeta_{i})_{i\in\mathbb{Z}}\in\Gamma^{1},$ such that $\zeta \notin \Gamma^{s}(\mathcal{M}).$ Since $T\mathcal{M}=E^{s}\oplus E^{u},$ for each $p\in M_{i}$ we have that $\zeta_{}(p)=v_{i}^{s}(p)+v_{i}^{u}(p)$, where $v_{i}^{s}(p)\in E^{s}$ and $v_{i}^{u}(p)\in E^{u}$. If $v_{i}^{u}(p)=0 $  for all $p$, then $\zeta \in \Gamma^{s}(\mathcal{M})$, consequently, there exists $p\in \mathcal{M}$ such that $v_{i}^{u}(p)\neq 0. $  Thus,
 \[\Vert \textbf{F}_{i}^{n} (\zeta_{i})(p)\Vert=  \Vert D\mathcal{F}_{i}^{n} (\zeta_{i} (p))\Vert\geq \Vert D\mathcal{F}_{i}^{n} (  v_{i}^{u}(p)  )\Vert - \Vert D\mathcal{F}_{i}^{n} (  v_{i}^{s}(p)  )\Vert \geq c^{-1}\lambda^{-n}\Vert v_{i}^{u}(p)\Vert -c\lambda^{n}\Vert v_{i}^{s}(p)\Vert.\]
This implies that $\limsup_{n\rightarrow +\infty}\Vert \textbf{F}^{n}(\zeta) \Vert \rightarrow \infty$ and therefore  $\limsup_{n\rightarrow +\infty}\Vert \textbf{F}^{n}(\zeta) \Vert ^{1/n}\geq 1$, which contradicts \eqref{2wew2}. 
Consequently,  $\zeta_{i}(p)=0$ for all $p\in \mathcal{M}$ and hence $\zeta \in \mathcal{M}$ and hence $\zeta \in \Gamma^{s}(\mathcal{M})$.  This proves that $\Gamma^{1}\subseteq \Gamma^{s}(\mathcal{M})$. Analogously we can prove that $\Gamma^{2}\subseteq \Gamma^{u}(\mathcal{M})$.  Therefore $ \Gamma(\mathcal{M})=\Gamma^{s}(\mathcal{M})\oplus \Gamma^{u}(\mathcal{M})$. It follows from Lemma \ref{lema44}
that $\mathcal{F}$ s.p.a.
Finally, by the definition of \textbf{F}, it is clear that \textbf{F} is bounded if and only if $\sup _{i\in \mathbb{Z}}\Vert D f_{i}\Vert = \infty$.
\end{proof}

Proposition \ref{multiplicativef2} provides of many examples of Anosov families  $  ( f_{i} )_{i\in\mathbb{Z}}$ with  
$\sup _{i\in \mathbb{Z}}\Vert D f_{i}\Vert = \infty$. A
question that arises from Theorem \ref{characte} is:

\begin{que} What kind of operator is \textbf{F}  if $ \mathcal{F}= ( f_{i} )_{i\in\mathbb{Z}}$ does not satisfy the property of angles and (or)
$\sup _{i\in \mathbb{Z}}\Vert D f_{i}\Vert = \infty?$\end{que}

\section{Openness of $\mathcal{A}^{2}_{b}(\mathcal{M})$}
 
As we said in  Introduction, in \cite{Jeo1} we  proved    $\mathcal{A}(\mathcal{M})$ is open in $\mathcal{D}^{1}(\mathcal{M})$ endowed with the strong topology, that is, for each 
$\mathcal{F}\in \mathcal{A}(\mathcal{M})$ there exists a sequence of positive numbers $ (\delta_{i})_{i\in\mathbb{Z}}$  such that  
$B^{1} (\mathcal{F}, (\delta_{i})_{i\in\mathbb{Z}})\subseteq \mathcal{A}(\mathcal{M}).$   In that case,   if we do not ask for any additional condition on $\mathcal{F}$, it is not always possible to take the sequence $\delta_{i}$ bounded away from zero, that is, $\delta_{i}$ could decay as $i\rightarrow \pm \infty$.    

\medskip

Set \[\mathcal{A}^{2}_{b}(\mathcal{M})=\{\mathcal{F}=(f_{i})_{i\in\mathbb{Z}}\in \mathcal{D}^{2}(\mathcal{M}): \mathcal{F} \text{ is Anosov, s.p.a.  and }\sup_{i\in\mathbb{Z}}\Vert Df_{i}\Vert_{C^{2}}<\infty\},\]
where \( \Vert \phi\Vert_{C^{2}}= \max\left\{ \Vert D \phi\Vert,  \Vert D\phi^{-1}\Vert,  \Vert D^{2}\phi\Vert, \Vert D^{2}\phi^{-1}\Vert \right\} \)   for a $ C^{2}$-diffeomorphism \(\phi.\)
The  goal of this section is to show  for any $\mathcal{F}=(f_{i})_{i\in\mathbb{Z}}\in \mathcal{A}^{2}_{b}(\mathcal{M})$, there exists a $\varepsilon>0$ such that: 
 \[  B^{2}(\mathcal{F},\varepsilon)=\{  \mathcal{G}\in \mathcal{D}^{2}(\mathcal{M}): d^{2}_{unif}(\mathcal{F},\mathcal{G})<\varepsilon \}\subseteq \mathcal{A}^{2}_{b}(\mathcal{M}).\]
 The basic neighborhood $B^{2}(\mathcal{F},\varepsilon)$ is called \textit{uniform}.

\medskip

From now on,    $\mathcal{F}$ will be an Anosov family in $\mathcal{A}^{2}_{b}(\mathcal{M})$ with constants $\lambda  \in (0,1)$ and $c \geq1$.  
 
\medskip

We will work using exponential charts (see Section 2). 
Take 
\[ S_{\mathcal{F}}=  \sup_{i\in \mathbb{Z}}\Vert D f_{i}\Vert_{C^{2}} \quad\text{ and }\quad r \in (0,\varrho/20S_{\mathcal{F}}).\] 
  Fix  $\mathcal{G}=(g_{i})_{i\in\mathbb{Z}}\in B ^{2}(\mathcal{F}, r) $. If $p,q\in M_{i}$  and $d(p,q)\leq r$, then we have \(d(f_{i}(p),g_{i}(q))<\varrho/2\) and \(  d(f_{i}^{-1}(p),g_{i}^{-1}(q))<\varrho/2.\) Therefore,
\( g_{i}(\overline{B(p,r)})\subseteq B(f_{i} (p),\varrho/2)\) and \( g_{i}^{-1}(\overline{B(f_{i}(p),r )})\subseteq B(p,\varrho/2).\) 
Consequently,  \begin{align*}\tilde{g}_{p}= &\, \text{exp}_{\textbf{\textit{f}}(p)}^{-1}\circ \mathcal{G}\circ\text{exp}_{p} :B (0_{p},r)\rightarrow B (0_{\mathcal{F}(p)},\varrho/2)
\\ 
 \text{and }\quad\tilde{g}_{p}^{-1} & =  \text{exp}_{p}^{-1}\circ \mathcal{G} ^{-1}\circ\text{exp}_{\mathcal{F}(p)} :B (0_{\mathcal{F}(p)},r)\rightarrow B (0_{p},\varrho/2),
\end{align*}
  are well-defined for each $p\in \mathcal{M}$.

\begin{propo}\label{primeirolem} For $\tau>0$, there exist $\tilde{r}>0$, $\tilde{\delta}>0$  and  $X_{i}=\{p_{1,i},\dots,p_{m_{i},i}\}\subseteq M_{i}$   for each $i\in\mathbb{Z}$, such that  $M_{i}=\cup _{j=1}^{m_{i}}B(p_{j,i},\tilde{r})$  for each $i\in\mathbb{Z}$ and, furthermore, for every $\mathcal{G}\in B^{2}(\mathcal{F},\tilde{\delta})$, we have that
\begin{equation*}\label{primeirolemt}\sigma(\tilde{r},\mathcal{G})= \sup_{p\in X_{i},i\in\mathbb{Z}} \left\{\sup_{z\in B(0_{p},\tilde{r})}\Vert D_{z}( D_{0} (\tilde{f}_{p})  - \tilde{g}_{p}) \Vert_{\star} ,\sup_{z\in B(0_{f_{i}(p)},\tilde{r})}\Vert D_{z}(D_{0} (\tilde{f}_{p}^{-1})  - \tilde{g}^{-1} _{p}) \Vert_{\star} \right\} <\tau.
\end{equation*}  
\end{propo}
 \begin{proof} Fix $p\in M_{i}.$ Let   $z\in B(0_{p},r )$ and $(v,w)\in E^{s}\oplus E^{u}.$  There exists $K>0$ (which does not depends on $p$), such that $$\Vert D_{0}   (\tilde{f_{p}})  (v,w)  -D_{z} (\tilde{f}_{p})  (v,w) \Vert_{\star} \leq K[1+\Vert Df_{i}\Vert_{\star}] \Vert D^{2}f_{i}\Vert_{\star}\Vert z\Vert _{\ast}\Vert (v,w)\Vert_{\star}$$ (see \cite{Liu2}, p. 50).  
 Thus, 
\begin{align*}\Vert D_{z}( D _{0}(\tilde{f}_{p})  - \tilde{g}_{p}) (v,w)\Vert_{\star}
&= \Vert D  _{0} (\tilde{f_{p}})  (v,w)  -D_{z}  (\tilde{g}_{p}) (v,w) \Vert_{\star} \\
&\leq  \Vert D_{0}   (\tilde{f_{p}})  (v,w)  -D _{z}(\tilde{f}_{p})  (v,w) \Vert_{\star}+\Vert D_{z}   (\tilde{f_{p}}-\tilde{g}_{p})  (v,w)  \Vert_{\star}\\
& \leq K[1+S_{\mathcal{F}}] S_{\mathcal{F}}\Vert z\Vert_{\star} \Vert (v,w)\Vert_{\star}+d^{2}(f_{i},g_{i})   \Vert(v,w)  \Vert_{\star}.
\end{align*} 
Therefore, if $\tilde{\delta}_{1}<\tau/2$ and $\tilde{r}_{1} < \tau/2( K(1+S_{\mathcal{F}})S_{\mathcal{F}})$, then  for each $\mathcal{G}\in B^{2}(\mathcal{F},\tilde{\delta}_{1})$ and $z\in B(0_{p},\tilde{r}_{1} ),$  we have    $  \Vert D_{z}( D_{0} (\tilde{f}_{p})  - \tilde{g}_{p}) \Vert< \tau$.  

Analogously we can prove   there exist    $\tilde{\delta}_{2}>0$ and $ \tilde{r}_{2}>0$ such that, if $\mathcal{G}\in B^{1}(\mathcal{F},\tilde{\delta}_{2})$ and $z\in B(0_{f_{i}(p)},\tilde{r}_{2})$,   then   $ \Vert D_{z}(D _{0}(\tilde{f}_{p}^{-1})  - \tilde{g}^{-1} _{p}) \Vert_{\star}\leq \tau$. 

Take $\tilde{r}=\min\{\tilde{r}_{1},\tilde{r}_{2}\}$ and $\tilde{\delta}=\min\{\tilde{\delta}_{1},\tilde{\delta}_{2}\}$.  Notice that neither $\tilde{r}$ nor $\tilde{\delta}$ depend on $p$.    Since $M_{i}$ is compact, we can choose a finite subset  $X_{i}=\{p_{1,i},\dots,p_{m_{i},i}\}\subseteq M_{i}$   such that $M_{i}=\cup _{j=1}^{m_{i}}B(p_{j,i},\tilde{r})$ for each $i\in\mathbb{Z}$, which proves the proposition. 
\end{proof}

From now on, we will fix  $\alpha\in (0,\frac{1-\tilde{\lambda}}{1+\tilde{\lambda}})$. Set
 \begin{equation*} 
 \sigma_{A}:=\min\left\{ \frac{(\tilde{\lambda}^{-1}-\tilde{\lambda})\alpha}{2(1+\alpha)^{2}},\frac{\tilde{\lambda}^{-1}(1-\alpha)-(1+\alpha)\alpha}{2(1+\alpha)}\right\}.
\end{equation*}
Furthermore, we will suppose that   $\tilde{r}$ and $\tilde{\delta}\leq \frac{1}{2}   \min\{\tilde{r},\sigma_{A} \}
  $  are small enough such that $\sigma(\tilde{r})<\sigma_{A}$.   Fix $p\in X_{i}$ for some $i\in \mathbb{Z}$. For $q\in B(p,\tilde{r})$, consider $z=\text{exp}_{p}^{-1}(q)$, \( \tilde{E}^{s}_{z}=D_{q}(\text{exp}_{p}^{-1}\circ \text{exp}_{q}) (E_{q}^{s})\), \( \tilde{E}^{u}_{z}=D_{q}(\text{exp}_{p}^{-1}\circ \text{exp}_{q}) (E_{q}^{u})\),    
\begin{align*}K_{\alpha,\mathcal{F},z}^{s} &=\{(v,w)\in \tilde{E}_{z}^{s}\oplus \tilde{E}_{z}^{u}: \Vert w\Vert_{\star}< \alpha \Vert v \Vert_{\star}\}\cup\{(0,0)\},  \\
\text{and }\quad K_{\alpha,\mathcal{F},z}^{u} &=\{(v,w)\in \tilde{E}_{z}^{s}\oplus \tilde{E}_{z}^{u}: \Vert v\Vert_{\star}< \alpha \Vert w\Vert_{\star}\}\cup\{(0,0)\} .\end{align*}

 The next two lemmas can be shown   analogously to the Lemmas 4.2 and 4.3 in \cite{Jeo1}, respectively.

\begin{lem}\label{limitadocones12} Let  $\mathcal{G}\in B^{1}(\mathcal{F},\tilde{\delta})$. For each $i\in\mathbb{Z}$, if $p\in X_{i}$ we have: 
\begin{enumerate}[i.]
\item  $D_{z} (\tilde{g}_{p}
) (\overline{K_{\alpha,\mathcal{F},z}^{u}})
\subseteq K_{\alpha,\mathcal{F},\tilde{g}_{p}(z)}^{u}$ for all   $ z\in B  (0_{p},\tilde{r})$, and
\item $
 D_{z}(\tilde{g}_{p})  ^{-1}(\overline{K_{\alpha,\mathcal{F},\tilde{g}_{p}(z)}^{s}})
\subseteq K_{\alpha,\mathcal{F},z}^{s}$ for all   $ z\in B  (0_{\mathcal{F}(p)},\tilde{r})$.
\end{enumerate}
See Figure \ref{conosinvariantes}.
\end{lem}

\begin{figure}[ht] 
\begin{center}

\begin{tikzpicture}
\draw[black!7,fill=black!10, ultra thin] (-7.9,-0.4)rectangle (-4.1,3.4);
\draw[black, fill=black!40, thin] (-6,1.5) -- (-7,3.4) -- (-5,3.4) -- cycle;
%\draw[black, fill=black!70, thin] (-6,1.5) -- (-6.8,3.4) -- (-5.4,3.4) -- cycle;
\draw[black, fill=black!40, thin] (-6,1.5) -- (-7,-0.4) -- (-5,-0.4) -- cycle;
%\draw[black, fill=black!70, thin] (-6,1.5) -- (-5.2,-0.4) -- (-6.6,-0.4) -- cycle;
\draw[black, fill=black!40, thin] (-6,1.5) -- (-4.1,2.5) -- (-4.1,0.5) -- cycle;
\draw[black, fill=black!70, thin] (-6,1.5) -- (-4.1,2.2) -- (-4.1,0.9) -- cycle;
\draw[black, fill=black!40, thin] (-6,1.5) -- (-7.9,2.5) -- (-7.9,0.5) -- cycle;
\draw[black, fill=black!70, thin] (-6,1.5) -- (-7.9,2.1) -- (-7.9,0.8) -- cycle;

\draw[black!7,fill=black!10, ultra thin] (-2.4,-0.4)rectangle (1.4,3.4);
\draw[black, fill=black!40, thin] (-0.5,1.5) -- (-1.5,3.4) -- (0.5,3.4) -- cycle;
\draw[black, fill=black!70, thin] (-0.5,1.5) -- (-1.3,3.4) -- (0.1,3.4) -- cycle;
\draw[black, fill=black!40, thin] (-0.5,1.5) -- (-1.5,-0.4) -- (0.5,-0.4) -- cycle;
\draw[black, fill=black!70, thin] (-0.5,1.5) -- (-1.1,-0.4) -- (0.3,-0.4) -- cycle;
\draw[black, fill=black!40, thin] (-0.5,1.5) -- (1.4,2.5) -- (1.4,0.5) -- cycle;
\draw[black, fill=black!40, thin] (-0.5,1.5) -- (-2.4,2.5) -- (-2.4,0.5) -- cycle;

\draw[<->] (-6,-0.5) -- (-6,3.5);
\draw[<->] (-8,1.5) -- (-4,1.5); 
\draw (-6.2,4.1) node[below] {\quad{\small $\tilde{E}_{z}^{u}$}}; 
\draw (-3.9,1.7) node[below] {\quad{\small $\tilde{E}_{z}^{s}$}};
\draw (-4.6,3.3) node[below] {\small $T_{z}M$};
\draw (-3.2,2.45) node[below] {\small $D(\tilde{g}_{p})_{q} ^{-1}$};
\draw[<-, very thick] (-4,1.8) -- (-2.5,1.8);

\draw[<->] (-0.5,-0.5) -- (-0.5,3.5);
\draw[<->] (-2.5,1.5) -- (1.5,1.5);
\draw (-0.7,4.1) node[below] {\quad{\small $\tilde{E}_{q}^{u}$}};
\draw (1.6,1.7) node[below] {\quad{\small $\tilde{E}_{q}^{s}$}};
\draw (0.7,3.3) node[below] {\quad{\small $T_{q}M$}};
\draw (-3.2,3.7) node[below] {\small $D (\tilde{g}_{p}
)_{z}$}; 
\draw[very thick, ->] (-5,3.6) .. controls (-3.5,3.8) and (-2.5,3.8) .. (-1,3.6);
 \draw (-7,4.1) node[below] {\small $K_{\alpha,\mathcal{F},z}^{u}$};
\draw (-8.7,2.7) node[below] {\quad{\small $K_{\alpha,\mathcal{F},z}^{s}$}};
\draw (0.7,4.1) node[below] {\small $K_{\alpha,\mathcal{F},q}^{u}$};
\draw (1.7,2.7) node[below] {\quad{\small $K_{\alpha,\mathcal{F},q}^{s}$}};
\end{tikzpicture} 
\end{center}
\caption{Stable and unstable invariant $\alpha$-cones. $q=\tilde{g}_{p}(z)$} \label{conosinvariantes}
\end{figure}
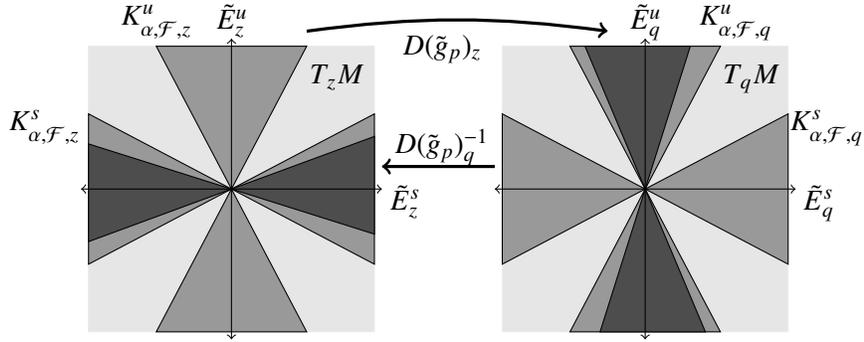
 
\begin{lem}\label{limitadocontract12} There exists $\eta\in (0,1)$ such that, 
if $\mathcal{G}\in B^{2}(\mathcal{F},\tilde{\delta})$ then,  for $p\in X_{i}$ and    $z \in B (0_{p},\tilde{r})$ we have
\begin{enumerate}[i.]
\item $\Vert D _{z}(\tilde{g}_{p}) (y)\Vert_{\star} \geq \eta^{-1}\Vert y \Vert_{\star}$ for each $y\in \overline{K_{\alpha,\mathcal{F},z}^{u}}$;
\item $\Vert   D_{z}(\tilde{g}^{-1} _{p}) (y)\Vert_{\star}\geq \eta^{-1}\Vert y\Vert_{\star}$ for each $y\in \overline{K_{\alpha,\mathcal{F},\tilde{g}_{p}(z)}^{s}}$.
\end{enumerate}
\end{lem}

Now, take $\xi\leq\frac{1}{2C}   \min\{\tilde{r},\sigma_{A} \}$ (see \eqref{metriceqe}).  Using the Lemmas \ref{limitadocones12} and \ref{limitadocontract12} we can prove that:

\begin{lem}\label{invariantesespacios}
Let $\mathcal{G}=(g_{i})_{i\in\mathbb{Z}}\in B^{2}(\mathcal{F}, \xi)$.
For each  $p\in \mathcal{M}$, take
 \begin{equation*}  F^{s}_{p}=\bigcap_{n=0}^{\infty}D_{\mathcal{G}^{n}(p)}(\mathcal{G}^{-n} )  (\overline{K_{\alpha,\mathcal{F},\mathcal{G}^{n}(p)}^{s}})\quad\text{ and }\quad F^{u}_{p}=\bigcap_{n=0}^{\infty}D_{\mathcal{G}^{-n}(p)}(\mathcal{G}^{n})  (\overline{K_{\alpha,\mathcal{F},\mathcal{G}^{-n}(p)}^{u}}).
 \end{equation*}
 The families  $F^{s}_{p}$ and $F^{u}_{p}$  are   $D\mathcal{G}$-invariant   subspaces and $T_{p}\mathcal{M}=F^{s}_{p}\oplus F^{u}_{p}$ for each $p\in \mathcal{M}$. Furthermore, there exist $\eta\in (0,1)$ and   $\tilde{c}\geq1$ such that,  with the splitting $T_{p}\mathcal{M}=F^{s}_{p}\oplus F^{u}_{p}$,   $\mathcal{G}$ is an Anosov family with constant $\eta$ and   $\tilde{c}$, which  s.p.a. (see Figure \ref{intercone}).   
\end{lem}
\begin{figure}[ht] 
\begin{center}

\begin{tikzpicture}
\draw[black!7,fill=black!10, ultra thin] (-7.9,-0.4)rectangle (-4.1,3.4);
\draw[black, fill=black!30, thin] (-6,1.5) -- (-7,3.4) -- (-5,3.4) -- cycle;
\draw[black, fill=black!50, thin] (-6,1.5) -- (-6.8,3.4) -- (-5.4,3.4) -- cycle;
\draw[black, fill=black!69, thin] (-6,1.5) -- (-6.6,3.4) -- (-5.8,3.4) -- cycle;
\draw[black, fill=black!85, thin] (-6,1.5) -- (-6.4,3.4) -- (-5.9,3.4) -- cycle;
\draw[black, fill=black!30, thin] (-6,1.5) -- (-7,-0.4) -- (-5,-0.4) -- cycle;
\draw[black, fill=black!50, thin] (-6,1.5) -- (-5.2,-0.4) -- (-6.6,-0.4) -- cycle;
\draw[black, fill=black!69, thin] (-6,1.5) -- (-6.3,-0.4) -- (-5.4,-0.4) -- cycle;
\draw[black, fill=black!85, thin] (-6,1.5) -- (-6.1,-0.4) -- (-5.6,-0.4) -- cycle;

\draw[black, fill=black!30, thin] (-6,1.5) -- (-4.1,2.5) -- (-4.1,0.5) -- cycle;
\draw[black, fill=black!50, thin] (-6,1.5) -- (-4.1,2.2) -- (-4.1,0.8) -- cycle;
\draw[black, fill=black!69, thin] (-6,1.5) -- (-4.1,2) -- (-4.1,1) -- cycle;
\draw[black, fill=black!85, thin] (-6,1.5) -- (-4.1,1.9) -- (-4.1,1.3) -- cycle;

\draw[black, fill=black!30, thin] (-6,1.5) -- (-7.9,2.5) -- (-7.9,0.5) -- cycle;
\draw[black, fill=black!50, thin] (-6,1.5) -- (-7.9,2.2) -- (-7.9,0.8) -- cycle;
\draw[black, fill=black!69, thin] (-6,1.5) -- (-7.9,2) -- (-7.9,1) -- cycle;
\draw[black, fill=black!85, thin] (-6,1.5) -- (-7.9,1.7) -- (-7.9,1.2) -- cycle;

 \draw[<->] (-6,-0.5) -- (-6,3.5);
 \draw[<->] (-8,1.5) -- (-4,1.5); 
\draw (-6.2,4) node[below] {\quad{\small $E_{p}^{u}$}}; 
\draw (-3.9,1.7) node[below] {\quad{\small $E_{p}^{s}$}};
\draw (-4.6,3.3) node[below] {\small $T_{p}M$};
\draw (-5.4,-0.5) node[below] {\small $F_{p,3}^{u}$};
 \draw[->] (-5.5,-0.6) -- (-5.7,-0.45);
\draw (-6,-0.5) node[below] {\small $F_{p,2}^{u}$};
\draw[->] (-6.2,-0.6) -- (-6.2,-0.45);
\draw (-6.6,-0.5) node[below] {\small $F_{p,1}^{u}$};
\draw[->] (-6.5,-0.6) -- (-6.5,-0.45);
\draw (-8.4,1.6) node[below] {\small $F_{p,3}^{s}$}; 
\draw[->] (-8.1,2.1) -- (-7.95,2.1);
\draw (-8.4,2) node[below] {\small $F_{p,2}^{s}$};
\draw[->] (-8.1,1.8) -- (-7.95,1.8);
\draw (-8.4,2.4) node[below] {\small $F_{p,1}^{s}$};
\draw[->] (-8.1,1.3) -- (-7.95,1.3);

\draw (-5,4) node[below] {\small $K_{\alpha,\mathcal{F},p}^{u}$};
\draw (-3.8,2.3) node[below] {\quad{\small $K_{\alpha\mathcal{F},p}^{s}$}};
 \end{tikzpicture} 

\end{center}
\caption{$F_{p,n}^{r}=\bigcap_{k=1}^{n}D_{\mathcal{G}^{\pm k}(p)} \mathcal{G}^{\pm k}   (\overline{K_{\alpha,\mathcal{F},\mathcal{G}^{\pm k}(p)}^{s}})$, for $r=s,u$ and $n=1,2,3$.} \label{intercone}
\end{figure}
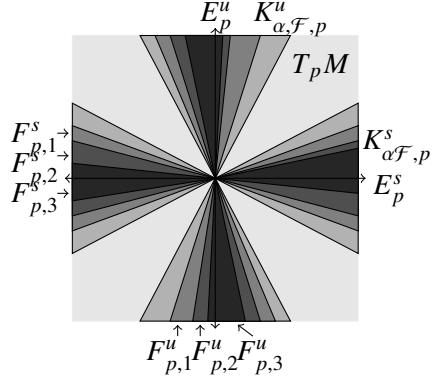
\begin{proof} See \cite{Jeo1}, Lemmas 4.4 and 4.5.
\end{proof}

Consequently,  we have:

 \begin{teo}\label{teoprincipal1} For all $\mathcal{F}\in \mathcal{A}^{2}_{b}(\mathcal{M})$, there exists $\xi>0$ such that $B^{2}(\mathcal{F}, \xi)\subseteq \mathcal{A}^{2}_{b}(\mathcal{M})$. That is,  $\mathcal{A}^{2}_{b}(\mathcal{M})$ is open in ($\mathcal{D}^{2} (\mathcal{M}),\tau_{unif}$). 
\end{teo}

\section{Local Stable and Unstable Manifolds for Anosov Families} 

In    \cite{Jeo2},  we  give conditions for obtain stable and unstable  manifolds at each point of each component $M_{i}$. In that case, the size of each submanifold could decay as $i\rightarrow \pm \infty$.   In  Theorems \ref{variedadeinstave} and \ref{variedadeestavel} we will see that each $\mathcal{F}\in\mathcal{A}^{2}_{b}(\mathcal{M})$  admits stable and unstable  manifold   with  the same size at each point.

   \medskip
   
    Given two points $p,q\in \mathcal{M}$, set
\begin{align*} 
\Theta_{p,q}&= \underset{n\rightarrow \infty}\limsup  \frac{1}{n}\log d(\mathcal{F}_{i}^{n}(q),\mathcal{F}_{i}^{n}(p))\quad  \text{and}\quad 
 \Delta_{p,q}   =\underset{n\rightarrow \infty}\limsup\frac{1}{n}\log d(\mathcal{F}_{i}^{-n}(q),\mathcal{F}_{i}^{-n}(p)).
\end{align*}

\begin{defin}\label{conjuntosestaviesfam}  Let  $\varepsilon>0$. Fix  $p\in \mathcal{M}$.    
\begin{enumerate}[\upshape (i)]
\item  \(\mathcal{W}^{s}(p,\varepsilon)=\{q\in B(p,\varepsilon):  \Theta_{p,q}<0\text{  and }\mathcal{F}_{i}^{\, n}(q)\in B(\mathcal{F}_{i}^{\, n}(p),\varepsilon)\text{ for }n\geq1\}:=\) the \textit{local stable set at} $p$; 
\item  \(\mathcal{W}^{u}(p,\varepsilon)=\{q\in B(p,\varepsilon):   
  \Delta_{p,q}<0\text{  and }\mathcal{F}_{i}^{-n}(q)\in B(\mathcal{F}_{i}^{-n}(p),\varepsilon)\text{ for }n\geq1\}:=\) the \textit{local unstable set at} $p$.
\end{enumerate}
\end{defin}

Take $\mathcal{F}\in\mathcal{A}^{2}_{b}(\mathcal{M})$ and $\tilde{\lambda}$  as in Remark  \ref{remarnormas}.
Fix $\alpha\leq(\tilde{\lambda}^{-1}-1)/2$ and \(\gamma\in (\tilde{\lambda}^{2},1)\). It follows from Proposition \ref{primeirolem} that    there exists $\delta>0$ such that  
 \begin{equation*}\label{ultimaeqw} \sigma(\delta,\mathcal{F})<\min\left\{  \frac{(\tilde{\lambda}^{-1}-\tilde{\lambda})\alpha}{(1+\alpha)^{2}} , \frac{(\gamma\tilde{\lambda}^{-1}-\tilde{\lambda})}{(1+\alpha)(1+\gamma)} , \frac{1-\tilde{\lambda}}{1+\tilde{\lambda}}\right\}.
  \end{equation*}

Thus, $\mathcal{F}$ satisfies the assumption from Proposition 3.3 in  \cite{Jeo2}.  Therefore, there exists   $\zeta\in (0,1)$ such that: 

\begin{teo}\label{variedadeinstave} For   each $p\in \mathcal{M}$,     $\mathcal{W}^{u}(p,\delta)$  is a differentiable submanifold of $\mathcal{M}$ and there exists $K^{u}>0$ such that:
 \begin{enumerate}[\upshape (i)]
\item   $T_{p}\mathcal{W} ^{u}(p,\delta)=E_{p} ^{u}$, 
\item $\mathcal{F}^{-1}(\mathcal{W}^{u}(p,\delta))\subseteq \mathcal{W}^{u}( \mathcal{F}  ^{\, -1}(p),\delta)$,  
\item if $q\in \mathcal{W} ^{u}(p,\delta)$ and $n\geq1$ we have  
\( d(\mathcal{F}^{-n}(q),\mathcal{F}^{-n}(p))\leq  K^{u}\zeta^{n}d(q,p).\)  
\item Let $(p_{m})_{m\in \mathbb{N}}$   be a  sequence in $ M_{i}$ converging to $p\in M_{i}$ as $m\rightarrow \infty$. If  $q_{m}\in   \mathcal{W} ^{u}(p_{m},\delta)$ converges to $q\in B (p,\delta)$ as $m\rightarrow \infty$,  then $q\in \mathcal{W}^{u}(p,\delta)$.  
\end{enumerate}
\end{teo}
\begin{proof}  See \cite{Jeo2}, Theorems 3.7 and 4.5. 
\end{proof}

 \begin{teo}\label{variedadeestavel}   For   each $p\in \mathcal{M}$,     $\mathcal{W}^{s}(p,\delta)$  is a differentiable submanifold of $\mathcal{M}$ and there exists $K^{s}>0$ such that:
\begin{enumerate}[\upshape (i)]
\item   $T_{p}\mathcal{W}^{s}(p,\delta)=E_{p} ^{s}$,   
\item $\mathcal{F}  (\mathcal{W} ^{s}(p,\delta))\subseteq \mathcal{W}^{s}(\mathcal{F}(p),\delta)$,
\item if $q\in \mathcal{W} ^{s}(p,\delta)$ and $n\geq1$ we have  
\(d(\mathcal{F}^{\, n}(q),\mathcal{F}^{\, n}(p))\leq  K^{s}\zeta^{n}d(q,p).
\)
\item Let $(p_{m})_{m\in \mathbb{N}}$   be a  sequence in $ M_{i}$ converging to $p\in M_{i}$ as $m\rightarrow \infty$. If  $q_{m}\in   \mathcal{W} ^{s}(p_{m},\delta)$ converges to $q\in B (p,\delta)$ as $m\rightarrow \infty$,  as $q\in \mathcal{W}^{s}(p,\delta)$.       
\end{enumerate}
\end{teo}
\begin{proof} See \cite{Jeo2}, Theorems 3.8 and 4.6. 
\end{proof}
 
$K^{u}$ and $K^{s}$ from Theorems \ref{variedadeinstave} and   \ref{variedadeestavel}, respectively, depend on the constant $ c$ of $\mathcal{F}$, on the constant $C$ in \eqref{metriceqe} and on the minimum angle between the stable and unstable subspaces of the splitting $T\mathcal{M}=E^{s}\oplus E^{u},$ which is positive because we are supposing that $\mathcal{F}$  s.p.a.

\section{Structural stability of  $\mathcal{A}^{2}_{b}(\mathcal{M})$}

 In this section   we will show that $\mathcal{A}^{2}_{b}(\mathcal{M})$ is uniformly structurally stable in $\mathcal{D}^{2}(\mathcal{M})$: for each $\mathcal{F}\in \mathcal{A}^{2}_{b}(\mathcal{M})$ there exists a uniform basic neighborhood $B^{2} (\mathcal{F}, \delta)$ of $\mathcal{F}$ such that,  
  each $\mathcal{G}\in B^{2} (\mathcal{F}, \delta)$ is   uniformly conjugate to $\mathcal{F}.$   Since $\Vert \cdot \Vert$ and $\Vert \cdot \Vert_{\star}$ (see Remark \ref{remarnormas}), then any
two-sided sequence $(h_{i} : M_{i}\rightarrow M_{i})_{i\in \mathbb{Z}}$  is equicontinuous with respect to the metric $\Vert \cdot\Vert$  if and only if
is equicontinuous with respect  $\Vert \cdot\Vert_{\star}$. Therefore, in order to show our result it is sufficient to prove it
considering the metric  $\Vert \cdot\Vert_{\star}$ on $\mathcal{M}$. Hence, we will consider the metric  $\Vert \cdot\Vert_{\star}$ given in \ref{metriceq}. To simplify
the notation, we will omit the symbol ``$  \star$''  for this metric. Therefore, we will fix $\mathcal{F}\in  \mathcal{A}^{2}_{b}(\mathcal{M})$ and we
can suppose that is strictly Anosov with respect to $\Vert \cdot\Vert$, with constant $\lambda\in (0, 1)$.
  
  \medskip
  
In order to prove the structural stability of Anosov families in $\mathcal{A}^{2}_{b}(\mathcal{M})$  we have adapted the Shub's ideas in
\cite{Shub} to prove the structural stability of Anosov diffeomorphisms on compact Riemannian manifolds.
We will divide the proof of this fact into a series of lemmas and propositions. Throughout this section,
we will consider $\tilde{r} > 0$, $\xi > 0$, $\varrho>0$ and $\eta \in [\lambda, 1)$ as in Section 5.

\medskip  
  
Let $0 < \varepsilon \leq \varrho/2$. We can identify $\mathcal{D}(\varepsilon)$   with $\Gamma_{\varepsilon}(\mathcal{M})$ (see Defifition \ref{subestp}) by the homeomorphism
\begin{align*} \Phi: \mathcal{D}(\varepsilon)&\rightarrow  \Gamma_{\varepsilon}(\mathcal{M})\\
(h_{i})_{i\in\mathbb{Z}}&\mapsto  (\Phi_{i}(h_{i}))_{i\in\mathbb{Z}} ,\end{align*}
where  $\Phi_{i}(h_{i})(p)=\text{exp}_{p}^{-1}(h_{i}(p))$, for $p\in M_{i}.$ Note that  $\Phi_{i}(I_{i})$ is the zero section in $\Gamma(M_{i} ).$

 \begin{lem}\label{definiciong1} Fix $\kappa>0$.    There exist $\xi^{\prime}\in (0,\xi]$ and $r^{\prime}\in(0,\tilde{r}/3]$ such that, if $ (g_{i})_{i\in\mathbb{Z}}\in   B^{2}(\mathcal{F},\tilde{\xi})$,  the map 
 \begin{equation*} \mathcal{G}:  \mathcal{D}(r^{\prime}) \rightarrow  \mathcal{D}(\kappa), \quad  
(h_{i})_{i\in\mathbb{Z}} \mapsto (g_{i-1} \circ h_{i} \circ f_{i-1}^{-1})_{i\in\mathbb{Z}}  
\end{equation*}
 is well-defined.
 \end{lem}
 \begin{proof} 
It is sufficient to prove  that there exist $\xi^{\prime}>0$ and $r^{\prime}>0$ such that, for every $i\in\mathbb{Z}$, if $g_{i}:M_{i}\rightarrow M_{i+1}$ is a diffeomorphism with $d^{2}(g_{i},f_{i})<\xi^{\prime}$ and  $h\in D(I_{i},r^{\prime})$,  then  $g_{i}  h f_{i} ^{-1}\in  D(I_{i+1},\kappa)$. For each continuous map $h:M_{i}\rightarrow M_{i}$, we have  
\[d(g_{i}  h f_{i} ^{-1}(p), I_{i+1}(p))\leq d(g_{i}  h f_{i} ^{-1}(p), g_{i}   f_{i} ^{-1}(p))+ d(g_{i} f_{i} ^{-1}(p), I_{i+1}(p)).\]
Let $\mathcal{S}:=\sup_{i\in\mathbb{Z}} \Vert Df_{i}\Vert <\infty$. If $r^{\prime}<\kappa/2(\mathcal{S}+1)$ and $\xi^{\prime} <\min\{1,\kappa /2\}.$ For   $i\in\mathbb{Z}$, if  $$g_{i}\in B^{1}(f_{i},\xi^{\prime})=\{g\in \text{Diff}^{1}(M_{i},M_{i+1}): \max\{d^{1}(g,f_{i}),d^{1}(g^{-1},f_{i}^{-1})<\xi^{\prime}\}$$ we have \[d(g_{i}   f_{i} ^{-1}(p), I_{i+1}(p))=d(g_{i}  f_{i} ^{-1}(p), f_{i} f_{i}^{-1}(p)) \leq d^{1}(g_{i},f_{i}) <\xi ^{\prime} <\kappa/2.\]
Furthermore, if  $h\in D(I_{i},r^{\prime})$ and $g_{i}\in B^{1}(f_{i},\xi^{\prime})$, then 
 \[d(g_{i}  h f_{i} ^{-1}(p), g_{i} f_{i}^{-1}(p))\leq \Vert D g_{i}\Vert d(h f_{i}^{-1}(p),f_{i}^{-1}(p))\leq (\xi^{\prime}+\mathcal{S})r^{\prime}< \kappa/2.\]   Therefore, if  $h\in D(I_{i},r^{\prime})$ and $ g_{i} \in   B^{1}(f_{i},\xi^{\prime})$, we have that  $g_{i}\circ h\circ f_{i}^{-1}\in D(I_{i+1},\kappa),$ which proves the lemma.
\end{proof}
 
Fix  \begin{equation}\label{zeta}\zeta \in (0,\min\{1-{\lambda},1-\eta,\varrho/2\}).\end{equation}
%Set \[ \overline{\Gamma}_{\zeta} (M_{i})= \Phi_{i}(D (I_{i},\zeta))= \{\sigma\in \Gamma(M_{i}): \sup_{x\in M_{i}}\Vert  \sigma(x)\Vert_{\ast} \leq \zeta  \}.\]
   It   follows from Lemma \ref{definiciong1} that there exist   \begin{equation}\label{dd}r^{\prime}\in (0,\tilde{r}/3)\quad\text{ and }\quad\xi^{\prime}\in (0,\min\{\xi,r^{\prime}(1-{\lambda} -\zeta)\}),
 \end{equation}  such that if $ (g_{i})_{i\in\mathbb{Z}}\in B^{2}(\mathcal{F},\xi^{\prime})$ then  
 \begin{align*}\textbf{G}:\Gamma_{r^{\prime}}(\mathcal{M})&\rightarrow \Gamma_{\zeta}(\mathcal{M}) \\
 (\sigma_{i})_{i\in\mathbb{Z}}&\mapsto (\Phi_{i}\mathcal{G}_{i-1}\Phi_{i-1}^{-1}(\sigma_{i-1}))_{i\in\mathbb{Z}}
 \end{align*}
 is well-defined, where $\mathcal{G}_{i-1}(h)=g_{i-1}\circ h \circ f_{i-1}^{-1},$ for $h\in D(I_{i-1},r^{\prime}).$ Consequently, 
 \[\Phi_{i} \mathcal{G}_{i-1} \Phi^{-1}_{i-1}(\sigma)(p)= \text{exp}_{p}^{-1}\circ g_{i-1}\circ \text{exp}_{f_{i-1}^{-1}(p)}\sigma (f_{i-1} ^{-1} (p)), \quad \text{for } p\in M_{i}, \sigma  \in \Gamma_{r^{\prime}}(M_{i-1}).\]  
 Let \textbf{F} be the operator given in Definition \ref{poerator} for $\mathcal{F}$. Since $\mathcal{F}\in\mathcal{A}^{2}_{b}(\mathcal{M})$,  \textbf{F} is a bounded hyperbolic
linear operator and $\mathcal{D}(\textbf{F}) = \Gamma(\mathcal{M})$  (see Theorem \ref{characte}).
 
\begin{lem}\label{lipchit} Take any $\zeta $ as in \eqref{zeta}. There exist   $\xi^{\prime}  \in (0,\xi] $ and $r^{\prime}\in (0,\tilde{r}/3]$ such that, for each $i\in\mathbb{Z}$, if $\mathcal{G}  \in   B^{1}(\mathcal{F},\xi^{\prime})$,  then   \[\text{Lip}([\textbf{F}  -\textbf{G}  ]  |_{\Gamma_{r^{\prime}} (\textbf{M})})<\zeta ,\]  where $\text{Lip}$ denotes a  Lipschitz constant. 
\end{lem}
  \begin{proof}
 For $\sigma\in\Gamma_{r^{\prime}}(M_{i-1})$ and $p\in M_{i}$ we have
 \[ [\textbf{F}_{i-1}-\Phi_{i}\mathcal{G}_{i-1}\Phi_{i-1}^{-1}](\sigma)(p)=[D_{f_{i-1}^{-1}(p)}(f_{i-1}) -\text{exp}_{p}^{-1}\circ g_{i-1} \circ \text{exp}_{f_{i-1}^{-1}(p)}](\sigma(f_{i-1}^{-1}(p))).\] 
 Let $q =f_{i-1}^{-1}(p)$. Note that
\[D_{q}( f_{i-1} )  = D_{0_{q}} (\text{exp}_{p}^{-1}\circ f_{i-1}\circ \text{exp}_{q}) \quad
\text{where $0_{q}$ is the zero vector in $T_{q} M_{i-1}$}.\]
For any $v \in T_{q} M_{i-1}$ with $\Vert v\Vert < \xi$, we have
 \begin{align*} D_{v}[D_{q}(f_{i-1}) -\text{exp}_{p}^{-1}\circ g_{i-1} \circ \text{exp}_{q}] & =  D_{v} [D_{0_{q}}(\text{exp}_{p}^{-1}\circ f_{i-1} \circ \text{exp}_{q}) -\text{exp}_{p}^{-1}\circ g_{i-1} \circ \text{exp}_{q}] \\
 &=  D_{0_{q}}(\text{exp}_{p}^{-1}\circ f_{i-1} \circ \text{exp}_{q}) -D_{v}(\text{exp}_{p}^{-1}\circ g_{i-1} \circ \text{exp}_{q}) .\end{align*}

As we saw in the proof of Proposition \ref{primeirolem},
\[\Vert  D_{0_{q}}(\text{exp}_{p}^{-1}\circ f_{i-1} \circ \text{exp}_{q}) -D_{v}(\text{exp}_{p}^{-1}\circ g_{i-1} \circ \text{exp}_{q}) \Vert \leq K [1+ \mathcal{S}_{\mathcal{F}}]\mathcal{S}_{\mathcal{F}}
\Vert v\Vert+ d^{2}(f_{i-1}, g_{i-1}). \]
Hence if  $\sigma=(\sigma_{i})_{i\in\mathbb{Z}}$,   $\tilde{\sigma}=(\tilde{\sigma}_{i})_{i\in\mathbb{Z}}\in \Gamma_{r^{\prime}}(M_{i-1})$ and $p\in  M_{i}$  we have
\begin{align*}\Vert [\textbf{F}_{i-1}-\Phi_{i}\mathcal{G}_{i-1}\Phi_{i-1}^{-1}](\sigma_{i-1})(p)-[\textbf{F}_{i-1}-\Phi_{i}\mathcal{G}_{i-1}\Phi_{i-1}^{-1}](\tilde{\sigma}_{i-1}) (p)\Vert &\leq \mathcal{J}\Vert (\sigma_{i-1}) (p) -(\tilde{\sigma}_{i-1}) (p)\Vert \\
&\leq \mathcal{J}\Vert \sigma_{i-1}    -\tilde{\sigma}_{i-1}   \Vert_{\Gamma_{i}}\\
&\leq \mathcal{J}\Vert \sigma  -\tilde{\sigma} \Vert_{\infty},\end{align*}
where  $\mathcal{J}=K [1+ \mathcal{S}_{\textbf{\textit{f}}}]\mathcal{S}_{\textbf{\textit{f}}}
\Vert v\Vert+ d^{2}(f_{i-1}, g_{i-1}).$
We can choose $ \xi^{\prime}$ and $r^{\prime}$ small enough such that $\mathcal{J}<\zeta.$  
Thus
 \[ \Vert [\textbf{F}_{i-1}-\Phi_{i}\mathcal{G}_{i-1}\Phi_{i-1}^{-1}](\sigma_{i-1}) -[\textbf{F}_{i-1}-\Phi_{i}\mathcal{G}_{i-1}\Phi_{i-1}^{-1}](\tilde{\sigma}_{i-1}) \Vert_{\Gamma_{i}}\leq \zeta\Vert \sigma  -\tilde{\sigma} \Vert_{\infty},\]
and therefore
 \[ \Vert [\textbf{F}- \textbf{G} ](\sigma)-[\textbf{F}- \textbf{G}](\tilde{\sigma})\Vert_{\infty}\leq \zeta\Vert \sigma  -\tilde{\sigma} \Vert_{\infty},\]
which proves the lemma.\end{proof}

From now on we will suppose that   $\xi^{\prime}$ and $r^{\prime}$ satisfy \eqref{dd} and Lemma    \ref{lipchit}. Furthermore, we will fix   $\mathcal{G} =(g_{i})_{i\in\mathbb{Z}}\in   B^{2}(\mathcal{F},\xi^{\prime})$.
 
 \begin{lem}\label{lemapuntofijo}  $\textbf{G}|_{\Gamma_{r^{\prime}}  (\mathcal{M})}$ has a fixed point in $  \Gamma_{r^{\prime}}  (\mathcal{M})$. 
 \end{lem} 
\begin{proof} Since $\Gamma (\mathcal{M}) = \Gamma ^{s}(\mathcal{M})  \oplus \Gamma ^{u}(\mathcal{M}) $, each $\sigma=(\sigma_{i})_{i\in\mathbb{Z}}\in   \Gamma_{r^{\prime}} (\mathcal{M})$ can be written as $\sigma= \sigma_{s}+\sigma_{u}$, where $\sigma_{s}=(\sigma_{i,s})_{i\in\mathbb{Z}}\in  \Gamma ^{s} (\mathcal{M})$ and $\sigma_{u}=(\sigma_{i,u})_{i\in\mathbb{Z}}\in  \Gamma^{u}(\mathcal{M}) $.  Let $ \tilde{\textbf{G}}$ be defined on $ \Gamma_{r^{\prime}} (\mathcal{M})$ as 
\[  \tilde{\textbf{G}}(\sigma)= (\textbf{G}(\sigma))_{s}+(\textbf{F}^{-1}[\sigma_{u}+\textbf{F}(\sigma_{u})  - (\textbf{G}(\sigma))_{u}])_{u}.\]
If $\sigma\in  {\Gamma}_{r^{\prime}}  (\mathcal{M})$ is a fixed point of $\tilde{\textbf{G}}$,  we have 
\( (\textbf{G}(\sigma))_{s}=\sigma_{s} \) and \( (\textbf{G}(\sigma))_{u}=\sigma_{u},\) 
that is, $\sigma$ is a fixed point of $\textbf{G}$. Therefore, in order to prove the lemma,  it is sufficient to find a fixed point of $\tilde{\textbf{G}}$. First we prove that $\tilde{\textbf{G}}$ is a contraction.  
 Let $\sigma=(\sigma_{i,s})_{i\in\mathbb{Z}}+(\sigma_{i,u})_{i\in\mathbb{Z}}$ and $\tilde{\sigma} =(\tilde{\sigma}_{i,s})_{i\in\mathbb{Z}}+(\tilde{\sigma}_{i,u})_{i\in\mathbb{Z}}$ in ${\Gamma}_{r^{\prime}}  (\mathcal{M}) $, where $\sigma_{s}, \tilde{\sigma_{s}}\in \Gamma ^{s}(\mathcal{M})$ and $\sigma_{u}, \tilde{\sigma_{u}}\in \Gamma ^{u}(\mathcal{M})$. Thus, we can prove that
 \begin{align*} \Vert  \tilde{\textbf{G}}(\sigma)_{i+1}-  \tilde{\textbf{G}} (\tilde{\sigma})_{i+1} \Vert_{\Gamma_{i+1}}  \leq  ({\lambda} +\zeta)\Vert \sigma -\tilde{\sigma}\Vert_{\infty}.
 \end{align*}
Hence, $\Vert  \tilde{\textbf{G}}(\sigma)-  \tilde{\textbf{G}} (\tilde{\sigma}) \Vert_{\infty}\leq ({\lambda} +\zeta)\Vert \sigma -\tilde{\sigma}\Vert_{\infty}.$ Since ${\lambda} +\zeta<1$ (see \eqref{zeta}), $\tilde{\textbf{G}}$ is a contraction. Now we prove that  $\tilde{\textbf{G}}({\Gamma}_{r^{\prime}}  (\mathcal{M}) )\subseteq {\Gamma}_{r^{\prime}}  (\mathcal{M}).$ If $0=(0_{i})_{i\in\mathbb{Z}}$ is the sequence of  the zero sections,   we have that 
\begin{align*}\Vert  \tilde{\textbf{G}}(\sigma)_{i+1}\Vert_{\Gamma_{i+1}} 
&\leq ({\lambda} +\zeta)\Vert \sigma \Vert_{\infty} + \max\{{\lambda}\Vert (\textbf{G}(0)_{u})_{i+1}  \Vert_{\Gamma_{i+1}},\Vert (\textbf{G}(0)_{s})_{i+1}\Vert _{\Gamma_{i+1}}\}, %\\&\leq (\tilde{\lambda} +\zeta)\Vert \sigma \Vert_{\infty}+ \Vert \textbf{G}(0) \Vert_{\infty}, 
\end{align*}
thus $\Vert  \tilde{\textbf{G}}(\sigma)\Vert_{\infty}\leq ({\lambda} +\zeta)\Vert \sigma \Vert_{\infty}+ \Vert \textbf{G}(0) \Vert_{\infty}$. 
Now, for each $i\in \mathbb{Z}$, $p\in M_{i}$, we have 
\begin{align*} \Vert \textbf{G}_{i}(0_{i})(p) \Vert &= \Vert \text{exp}_{p}^{-1}\circ g_{i}\circ \text{exp}_{f_{i}^{-1}(p)}(0_{f_{i}^{-1}(p)} )\Vert =\Vert \text{exp}_{p}^{-1}( g_{i}  f_{i}^{-1}(p) )\Vert\\
&= d ( g_{i}  f_{i}^{-1}(p),p)= d ( g_{i}  f_{i}^{-1}(p),f_{i}  f_{i}^{-1}(p))<\xi^{\prime}
\end{align*}
Consequently, if $\sigma\in  {\Gamma}_{r^{\prime}}  (\mathcal{M}) $, then $\Vert  \tilde{\textbf{G}}(\sigma)\Vert_{\infty}<({\lambda} +\zeta)r^{\prime} +\xi^{\prime}<r^{\prime}, $ that is, $\tilde{\textbf{G}}(\sigma)\in  {\Gamma}_{r^{\prime}} (\mathcal{M}) $. Therefore, $ \tilde{\textbf{G}}$ has a fixed point in ${\Gamma}_{r^{\prime}} (\mathcal{M}) $.
\end{proof}

\begin{lem}\label{prmerlema1}   Let $p\in \textbf{M}$ and $v=(v_{s},v_{u}),w=(w_{s},w_{u})\in B^{s}(0,\tilde{r})\times B^{u}(0,\tilde{r})$. If $\Vert v_{s}-w_{s}\Vert \leq \Vert v_{u}-w_{u}\Vert  $, then 
\[\Vert (\tilde{g}_{p}(v))_{s}-(\tilde{g}_{p}(w))_{s}\Vert \leq (\eta^{-1}- \zeta) \Vert v_{u}- w_{u}\Vert\leq   \Vert (\tilde{g}_{p}(v))_{u}-(\tilde{g}_{p}(w))_{u}\Vert.\]
On the other hand, if $\Vert v_{u}-w_{u}\Vert  \leq \Vert v_{s}-w_{s}\Vert$, then 
\[ \Vert (\tilde{g}_{p}^{-1}(v))_{u}-(\tilde{g}_{p}^{-1}(w))_{u}\Vert \leq (\eta^{-1}- \zeta) \Vert v_{s}- w_{s}\Vert\leq   \Vert (\tilde{g}_{p}^{-1}(v))_{s}-(\tilde{g}_{p}^{-1}(w))_{s}\Vert.\] 
\end{lem}
\begin{proof}See \cite{Shub}, Lemma II.1. 
\end{proof}

Let $v=(v_{s},v_{u})\in B^{s}(0,\tilde{r})\times B^{u}(0,\tilde{r})$ and  $q=\text{exp}_{p}(v)$.  Suppose  that, for $n=\pm 1$, $d(g^{n}(p),g^{n}(q))< \tilde{r}$. Thus, if $\Vert v_{s} \Vert \leq \Vert v_{u}\Vert$, by Lemma   \ref{prmerlema1} we have
\begin{align*}d(q ,p)&=\Vert v\Vert \leq \Vert v_{s}\Vert+  \Vert v_{u}\Vert\leq 2 \Vert v_{u} \Vert\leq 2(\eta^{-1}-\zeta)^{-1}\Vert (\tilde{g}_{p}(v))_{u}-(\tilde{g}_{p}(0_{p}))_{u}\Vert\\
&  \leq 2(\eta^{-1}-\zeta)^{-1}[\Vert (\tilde{g}_{p}(v))_{u}-(\tilde{g}_{p}(0_{p}))_{u}\Vert+\Vert (\tilde{g}_{p}(v))_{s}-(\tilde{g}_{p}(0_{p}))_{s}\Vert]\\
&\leq 2\sqrt{2}(\eta^{-1}-\zeta) ^{-1}\Vert \tilde{g}_{p}(v) -\tilde{g}_{p}(0_{p})\Vert\leq   2\sqrt{2}(\eta^{-1}-\zeta)^{-1}\tilde{r}. 
\end{align*}
Analogously   if $\Vert v_{u} \Vert  \leq \Vert v_{s}\Vert$, then \(d(q ,p)\leq 2\sqrt{2}(\eta^{-1}-\zeta)^{-1}\tilde{r}.\)  Inductively, we can prove that: 
  
  \begin{propo}\label{propodelacota} For each $i\in\mathbb{Z}$, if  $p,q\in M_{i}$ and $d(g_{i}^{n}(p),g_{i}^{n}(q))<\tilde{r}$ for each $n\in [-N,N]$, then \[d(q ,p)\leq 2\sqrt{2}(\eta^{-1}-\zeta)^{-N}\tilde{r}.\] 
\end{propo}

 Finally, we have:
 
\begin{teo}\label{teofprlf} $ \mathcal{A}^{2}_{b}(\mathcal{M})$  is uniformly structurally stable in $ \mathcal{F}^{2}(\mathcal{M})$, that is, any  $\mathcal{F}\in \mathcal{A}^{2}_{b}(\mathcal{M})$  is uniformly
structurally stable.\end{teo}

\begin{proof} Take $(g_{i})_{i\in\mathbb{Z}}\in B^{2}(\mathcal{F}, \xi^{\prime})$. It follows from Lemma \ref{lemapuntofijo} that there exists $\mathcal{H}=(h_{i})_{i\in\mathbb{Z}}$, with $h_{i}\in D (I_{i},\tilde{r}/3)$ for each $i$, such that \( g_{i}\circ h_{i}=h_{i+1}\circ f_{i}\) for each \(i\in\mathbb{Z}.\)   
We will prove that $\mathcal{H}$ is equicontinuous and each $h_{i}$ is injective.   Let $\alpha>0$. Take $N>0$ such that    $2\sqrt{2}(\eta^{-1}-\zeta)^{-N}\tilde{r}<\alpha$. Since   $(f_{i})_{i\in\mathbb{Z}}$ is equicontinuous, the family of  sequences    $$ \{(f_{i}^{n})_{i\in\mathbb{Z}}: n\in [-N,N]\}$$ is equicontinuous. Consequently, there exists $\beta>0$ such that, for each $i\in\mathbb{Z}$, if $x,y\in M_{i}$ and $d(x,y)<\beta,$ then $d(f_{i}^{n}(x),f_{i}^{n}(y))<\tilde{r}/3$ for any $n\in [-N,N]$. Hence, for each $i\in\mathbb{Z}$ and $n\in [-N,N]$,  if $x,y\in M_{i}$ and $d(x,y)<\beta,$ then
\begin{align*} d(g_{i}^{n}\circ h_{i}(x),g_{i}^{n}\circ h_{i}(y)) &\leq d(g_{i}^{n}\circ h_{i}(x),f_{i}^{n}(x))+d(f_{i}^{n}(x),f_{i}^{n}(y))+d(f_{i}^{n}(y),g_{i}^{n}\circ h_{i}(y)) \\
&=d(h_{i+n}\circ f_{i}^{n}(x),f_{i}^{n}(x))+d(f_{i}^{n}(x),f_{i}^{n}(y))+d(f_{i}^{n}(y),h_{i+n}\circ f_{i}^{n}(y)) \\  
&\leq \tilde{r}/3+\tilde{r}/3 +\tilde{r}/3=  \tilde{r}.
\end{align*}
It follows from Proposition \ref{propodelacota} that  $d(h_{i}(x),h_{i}(y))<\alpha$. This fact proves that $(h_{i})_{i\in\mathbb{Z}}$ is an equicontinuous family. Note that if $h_{i}(x)=h_{i}(y)$ for some $x,y\in M_{i}$, then $d(g_{i}^{n}\circ h_{i}(x),g_{i}^{n}\circ h_{i}(y))< \tilde{r} $ for any $n\in\mathbb{Z}$. Thus  $x=y$ and therefore $h_{i}$ is injective.     Analogously we can prove that $(h_{i}^{-1})_{i\in\mathbb{Z}}$ is equicontinuous. Consequently, $\mathcal{A}^{2}_{b}(\mathcal{M})$ is uniformly structurally stable in $\mathcal{D}^{2}(\mathcal{M})$. 
\end{proof}

It follows directly from Theorems \ref{teoprincipal1}  and   \ref{teofprlf} that if $\mathcal{F}=(f_{i})_{i\in\mathbb{Z}}$ is an Anosov family satisfying the property of angles consisting of a finite sequence of diffeomorphisms, there exists $\xi>0$ such that, if $\mathcal{G}\in  B^{2}(\mathcal{F}, \xi)$, then  $\mathcal{G}\in\mathcal{A}^{2}(\mathcal{M})$ and is uniformly conjugate to $\mathcal{F}.$

%  The author would like to thank the institution   Universidade de S\~ao Paulo (USP)   and the agencies CAPES and CNPq for their   hospitality and support  during the course of the writing. Special thanks for A. Fisher, my doctoral advisor, who has inspired and aided me along the way. 

\end{document}